\documentclass[11pt,reqno]{amsart}
\usepackage{diagram}
\usepackage{amssymb}
\usepackage{cite}
\usepackage[pagewise]{lineno}
\usepackage{appendix}
\usepackage{xcolor}
\numberwithin{equation}{section}

\theoremstyle{definition}
\newtheorem{thm}{Theorem}[section]
\newtheorem{cor}[thm]{Corollary}
\newtheorem{lem}[thm]{Lemma}
\newtheorem{exa}[thm]{Example}
\newtheorem{prop}[thm]{Proposition}
\newtheorem{defi}[thm]{Definition}
\newtheorem{rem}[thm]{Remark}

\newtheorem*{shc}{Singular Hodge conjecture}
\newtheorem*{hshc}{Homological Singular Hodge conjecture}
\newtheorem*{lsa}{Lefschetz standard conjecture A}

\DeclareMathOperator{\Ima}{\mathrm{Im}}

\DeclareMathOperator{\p3}{\mathbb{P}^3}

\DeclareMathOperator{\pr}{\mathrm{pr}}

\DeclareMathOperator{\Spec}{\mathrm{Spec}}

\DeclareMathOperator{\N}{\mathcal{N}}

\DeclareMathOperator{\mo}{\mathcal{O}}
\DeclareMathOperator{\x}{\mathcal{X}}

\newcommand{\mbb}[1]{\mathbb{#1}}
\newcommand{\mb}[1]{\mathbb{#1}}
\newcommand{\mc}[1]{\mathcal{#1}}
\newcommand{\mr}[1]{\mathrm{#1}}
\newcommand{\ov}[1]{\overline{#1}}
\newcommand{\wt}[1]{\widetilde{#1}}
\newcommand{\mf}[1]{\mathfrak{#1}}

\begin{document}

\title{Hodge Conjecture via Singular Varieties}

\author[A. Dan]{Ananyo Dan}

\address{CUNEF Universidad, C. de Leonardo Prieto Castro, 2, Moncloa - Aravaca, 28040 Madrid, Spain}

\email{dan.ananyo@cunef.edu}

\author[I. Kaur]{Inder Kaur}

\address{School of Mathematics \& Statistics, University of Glasgow, Glasgow G12 8QQ, U.K.}

\email{inder.kaur@glasgow.ac.uk}

\subjclass[2020]{$14$C$15$, $14$C$30$, $32$S$35$, $32$G$20$, $14$D$07$, $14$C$05$}

\keywords{Hodge conjecture,  Limit mixed Hodge
structures, Operational Chow group, Cycle class map, Singular varieties}

\date{\today}

\begin{abstract}
In this article we study the cohomological and homological (due to Jannsen) Hodge conjecture for singular varieties. The motivation for studying singular varieties comes from the fact that any smooth projective variety $X$ is birational to a (possibly singular) hypersurface $Y$ in a projective space. We prove that odd dimensional hypersurfaces with $A_n$ singularities satisfy both versions of the conjecture and moreover their (smooth) resolutions satisfy the classical Hodge conjecture, thus producing new examples of smooth varieties satisfying the classical Hodge conjecture.


\end{abstract}

\maketitle

\section{Introduction}
The underlying field will always be $\mbb{C}$.
Given a smooth, projective variety $X$, the classical \emph{Hodge conjecture in codimension} $p$ claims that every rational Hodge class in $H^{2p}(X,\mb{Q})$ is a cohomology class of an algebraic cycle of codimension $p$. The conjecture is largely open.
To the best of our knowledge there are four general strategies to 
study this conjecture. First is to use the Lefschetz hyperplane theorem in the case when $X$ is 
an ample divisor inside a smooth, projective variety which already satisfies the Hodge conjecture. 
In this case, $X$ satisfies the Hodge conjecture in all codimensions, except for the middle.
The second method is used when $X$ is a (smooth) blow down of a smooth, projective variety satisfying the Hodge conjecture, for example when 
$X$ is uniruled (see \cite{lewis, conmur}). The third method is applied mainly in the case of Fermat hypersurfaces \cite{shio}.
The strategy is to describe the Hodge classes using characters of a finite commutative group acting on the variety, thereby reducing 
the Hodge conjecture to a purely arithmetic statement. The fourth method is used in the case of abelian varieties \cite{Matt} and moduli spaces of (parabolic)
bundles over curves \cite{bishod}. The idea is to show that for a very general such variety, the Hodge classes are simply 
powers of the first Chern class of a very ample line bundle. Besides these four general methods, there are equivalent formulations 
of the Hodge conjecture \cite{dehod1, thohod1,bros} and methods that apply to very specific setups \cite{schoen, ink}.

In this article we take a different approach. Our goal is to reduce the Hodge conjecture to an analogous problem on 
a simpler, possibly singular variety. In \cite[Conjecture $7.2$]{jann} Jannsen conjectured the following homological Hodge conjecture for singular varieties:

 \begin{hshc}(Jannsen):
 Let $X$ be a complex variety. Then, for all $p \ge 0$, the cycle class map:
    \begin{equation}\label{eq:cyc-hom}
        \mr{cl}_p^{\mr{hom}}: A_p(X) \otimes \mb{Q} \to (2\pi i)^{-p} W_{-2p} H_{2p}^{\mr{BM}}(X,\mb{Q}) \cap F^{-p} H_{2p}^{\mr{BM}}(X,\mb{C}),\, \mbox{ is surjective},
    \end{equation}
    where $H_*^{\mr{BM}}(-)$ denotes the Borel-Moore homology and $A_p(X)$ is the usual Chow group.
    \end{hshc}
      
 This was formulated by Jannsen in \cite[Conjecture $7.2$]{jann}. Lewis in \cite{lewsing2} proved special cases of this conjecture. 
 However, it is not clear whether this can be used to produce new examples of varieties satisfying the 
 Hodge conjecture. More recently, de Cataldo, Migliorini and Musta{\c{t}}{\u{a}} proved the above conjecture for simplicial toric varieties (see \cite{decomb}).

    The cohomological counterpart for singular varieties is significantly more complicated (as mentioned by Jannsen \cite[p. 108]{jann}) since
    the usual Chow group is not functorial under arbitrary pullbacks. 
    However, thanks to Bloch-Gillet-Soul\'{e} \cite{soub}, (there is a functorial) cycle class map for singular, projective varieties $X$:
\begin{equation}\label{eq:cycle}
    \mr{cl}^p: A^p(X) \to  \mr{Gr}^W_{2p} H^{2p}(X,\mb{Q}) \cap F^p \mr{Gr}^W_{2p} H^{2p}(X,\mb{C})
\end{equation}
where $A^p(X)$ denotes the codimension $p$ operational Chow group (see \cite[\S $17$]{fult}), $\mr{Gr}^W_{2p} H^{2p}(X,\mb{Q})$ is the $2p$-th weight graded piece of the rational cohomology group which is known to be a pure Hodge structure of weight $p$ and $F^p$ denotes the $p$-th Hodge filtration of this graded piece. It was shown by Totaro in \cite{totchow} that $\mr{cl}^p$ does not lift to $H^{2p}(X,\mb{Q})$.  If $X$ is non-singular, then $\mr{cl}^p$ coincides with the usual cycle map. The natural cohomological Hodge conjecture for singular varieties can be stated as follows:
\begin{shc}
 Let $X$ be a projective variety. 
 Then $X$ is said to satisfy the \emph{singular Hodge conjecture in codimension} $p$, denoted $\mr{SHC}(p)$, 
 if the cycle class map \eqref{eq:cycle} is surjective. If $X$ is non-singular, then we will say that $X$ \emph{satisfies} $\mr{HC}(p)$.
\end{shc}

It is natural to ask why would one care about proving the singular Hodge conjecture when the classical (smooth) Hodge conjecture still remains wide open? Our motivation for doing this is two-fold: first it is interesting in its own right since such a statement has not been studied before. 
The much easier case of $p=1$ (i.e. singular Lefschetz $(1,1)$), has been studied by the authors in \cite{dan-chow}.
Secondly (as we show) we can produce examples of varieties satisfying the classical Hodge conjecture via singular varieties. 
An obvious strategy to producing new examples of varieties satisfying the classical Hodge conjecture may be to start with a previously 
known example and check if the blow-up along a subvariety again satisfies the classical Hodge conjecture. However, it is not known 
whether the resulting exceptional divisor will always satisfy the classical Hodge conjecture. In other words, it is unknown whether 
the property of satisfying the classical Hodge conjecture is preserved under blow-ups.  In comparison, we demonstrate in the proof 
of Theorem \ref{thm:final} that if we are willing to work with the singular Hodge conjecture, we can produce new  examples of HC 
from known examples of SHC by blowing up points where the associated exceptional divisor is not, in general, smooth. This 
motivates the idea that SHC may be a better notion to consider.

One of the main results of this article is that hypersurfaces with $A_{n}$ singularities  satisfy the singular Hodge conjecture, the homological Jannsen conjecture and their resolution (which are smooth varieties) actually satisfies the original Hodge conjecture. As mentioned earlier, all smooth, projective varieties are birational to (possibly singular) hypersurfaces. Moreover, $A_n$ singularities are one of the most commonly occurring singularities, so it makes sense to understand this case fully. We prove the following:

\begin{thm}\label{thm:intro-1}
Let $Y$ be a smooth, projective variety of dimension $2m$ satisfying the Hodge conjecture, $X \subset Y$ very ample divisor of $Y$ with at worst $A_n$ singularities. Then $X$ satisfies the singular Hodge conjecture.
\end{thm}

See Theorem \ref{thm:final} for a proof and Theorems \ref{th01} for a more general result. 
We remark that in \cite{schoen}, Shoen shows that a desingularization of a hypersurface $X$ in $\mb{P}^{2m}$ with at worst ordinary 
double point singularities (i.e., $A_1$-singularities) in a special configuration satisfies the Hodge conjecture. 

One of the key ingredients to proving Theorem \ref{thm:intro-1}, is to first show that for $A_n$ singularities, satisfying the singular Hodge conjecture is a birational property in the sense that if there exists a (possibly reducible) algebraic scheme satisfying the singular Hodge conjecture and 
birational to the $A_n$ singular variety, then the latter also satisfies the singular Hodge conjecture. Therefore, we can reduce the problem to proving the conjecture for a simpler variety. For this we embed our hypersurface with $A_n$ singularities as the central fibre of a smooth family, 
take the semi-stable reduction and then study the Hodge theoretic properties inherited by the central fiber from the smooth ones.
More precisely, we require that the $p$-th Hodge classes in the family satisfy a limiting property (see \S \ref{sec:mt-def} for the precise definition). We will call such a family, a \emph{Mumford-Tate family}.
Examples of well-known Mumford-Tate families include: families of odd dimensional hypersurfaces,  products of hypersurfaces, families of Hilbert schemes of points 
 (see Theorem \ref{thm:intro-3} below). 
 We show that the property of a family being Mumford-Tate is preserved under several well-known operations and thus we can easily produce new Mumford-Tate families from old one. In particular, we prove:

\begin{thm}\label{thm:intro-3}
Let $W$ be a smooth, projective variety of dimension $n$ and 
     \[\pi_1: \mc{X} \to \Delta^*, \, \pi_2: \mc{Y} \to \Delta^*\]
    be two smooth, projective families of ample hypersurfaces in $W$ in the sense that for all $t \in \Delta^*$, $\mc{X}_t$ and $\mc{Y}_t$
    are ample hypersurfaces in $W$. Then, the following holds true: 
    \begin{enumerate}
        \item If $n$ is even then $\pi_1, \pi_2$ are Mumford-Tate families.
        \item If $\pi_1$ and $\pi_2$ are Mumford-Tate families, then $\pi_1 \times \pi_2$ are Mumford-Tate families.
        \item If $\pi_1$ is a Mumford-Tate family, then the associated family of Hilbert scheme of points of length $2$ is a Mumford-Tate family.
        \item Let $Z \subset \mc{X}$ be a section of $\pi_1$ i.e., the composition $\pi_1:Z \to \mc{X} \to \Delta^*$ is an isomorphism.
        If $\mc{X}$ is a Mumford-Tate family, then the blow-up of $\mc{X}$ along $Z$ is a Mumford-Tate family.
        \end{enumerate}
\end{thm}

See Theorem \ref{thm:mt-exa} for a proof. More generally, if a smooth, projective family arises from a representable functor defined over a 
Mumford-Tate family, then the resulting family is again Mumford-Tate (see Theorem \ref{thm:lim-mt}). 
For example, moduli space of sheaves over Mumford-Tate families of surfaces is again Mumford-Tate (see Example \ref{sec:mod}). 
Similarly, families of Fano varieties of lines associated to Mumford-Tate families of cubic hypersurfaces is a Mumford-Tate family 
(see Example \ref{sec:fano}).

{\bf{Application to the classical Hodge conjecture and Jannsen's conjecture}}: The above theorems lead us to conclude the following results about Jannsen's conjecture and the classical Hodge conjecture. 

\begin{thm}\label{thm:intro-2}
Let $Y$ be a smooth, projective variety of dimension $2m$ satisfying the Hodge conjecture, $X \subset Y$ very ample divisor of $Y$ with at worst 
$A_{2k}$-singularities. Then, 
\begin{enumerate}
\item $X$ satisfies Jannsen's conjecture, 
\item the resolution $\wt{X}$ obtained by blowing-up the singular points (repeatedly) of $X$, satisfies the classical Hodge conjecture.
\end{enumerate}
\end{thm}

See Theorem \ref{thm:final} and Corollary \ref{cor:jann} for the proof. 
Furthermore, we also show that for $A_{2k+1}$-singularities, the Lefschetz standard conjecture A implies the Hodge conjecture (see Corollary \ref{cor:n-odd}).
Finally, although we do not give the proof to avoid repetitiveness, we also expect that Theorem \ref{thm:intro-2} can be extended to more general singular varieties.
This is because the two key ingredients
for the proof of the theorem are: a birational invariance of SHC for the singular variety as mentioned above and
the existence of Mumford-Tate families containing the given singular variety as the central fiber. 
The first condition is satisfied by several other singularities (besides $A_n$-singularities) such as 
 $D_n$-singularities, ``general" hypersurface singularities, etc. The second condition is satisfied not only by odd dimensional hypersurfaces, but 
 also even dimensional varieties arising as degeneration of product of odd dimensional hypersurfaces (see Theorem \ref{thm:intro-3}).
 One can also drop the assumption of \emph{very ample hypersurface} since blow-ups of Mumford-Tate families are again Mumford-Tate.

{\bf{Notation}:} Given a projective variety $X$, we will denote by $A^p(X)$ the $p$-th operational Chow group, by $A_p(X)$ the usual Chow group 
generated by codimension $p$ closed integral subvarieties and by $\mr{cl}^p$ ($\mr{cl}_p$) the cycle class map from $A^p(X)$ (resp. $A_p(X)$)
to $\mr{Gr}^W_{2p}H^{2p}(X,\mb{Q})$ (resp. $W_{-2p}H_{2p}(X,\mb{Q})$).

\vspace{0.2 cm}

 $\bf{Outline}$: The paper is organised as follows: in \S \ref{sec:prelim}
 we recall the Bloch-Gillet-Soul\'{e} cycle class map and
 compute the kernel of the specialization map. This calculation is used in the proof of Theorem \ref{thm:intro-1}.
  In \S \ref{sec:mt}, we introduce Mumford-Tate families
  and give examples. 
  In the final section \S \ref{sec:exa}, we apply these results to produce new examples of projective varieties satisfying the (singular) Hodge conjecture
  and the Jannsen's conjecture.
 In the appendices \ref{app:mon} and \ref{app:op}, we recall basic properties of the monodromy operators and the  operational Chow group, respectively.

 {\bf{Acknowledgements}:}
This article was motivated by some questions asked by
Prof. C. Simpson, after the second author gave a talk on the article \cite{hodc} at the workshop `Moduli of bundles and related structures' held at ICTS, Bengaluru, India. We thank Prof. Simpson for his interest and the organisers for organising the workshop. We are also grateful to Prof. C. Voisin for her encouragement.

 \section{Specialization and mixed Hodge structures}\label{sec:prelim}
 In this section we briefly recall some of the basics on limit mixed Hodge  structures (see \cite[\S $11$]{pet} for details), 
 operational Chow group (see \cite[\S $17$]{fult} for details) and the Bloch-Gillet-Soul\'{e}
 cycle class map (see \cite[Appendix]{soub}).

  \subsection{Setup}\label{se01}
Consider a smooth, projective family of varieties,
 \[\pi:\mc{X} \to \Delta^*,\]
 where $\Delta^*$ is the punctured unit disc. Denote by 
 $e: \mf{h} \to \Delta^*$
 the universal covering of $\Delta^*$, 
 where $\mf{h}$ denotes the upper half plane.
 Denote by  
$\mc{X}_\infty:=\mc{X}_{\Delta^*} \times_{\Delta^*} \mf{h}$ the base change of $\mc{X}$
to $\mf{h}$ via the exponential map $e$.
For any $s \in \mf{h}$, there is a natural inclusion of fiber:
\[i_s: \mc{X}_{s} \hookrightarrow \mc{X}_\infty.\]
Since $\mf{h}$ is simply connected, this inclusion induces an isomorphism of cohomology groups:
\begin{equation}\label{eq:iden}
    i_s^*: H^{2p}(\mc{X}_\infty, \mb{Z}) \xrightarrow{\sim}
 H^{2p}(\mc{X}_{e(s)}, \mb{Z}).
\end{equation}
  Note that, the morphism $i_s^*$ changes even if $e(s)$ does not.

\subsection{Monodromy and limit mixed Hodge structures}\label{sec:lhf}
The monodromy operator 
\[T:H^{2p}(\mc{X}_\infty, \mb{Z}) \to H^{2p}(\mc{X}_\infty, \mb{Z})\]
is the automorphism induced via pullback $h^*$ by the morphism 
\[h: \mf{h} \to \mf{h}\, \mbox{ sending } u \mapsto u+1.\]
Recall that for any $u \in \mf{h}$, the monodromy operator $T$ coincides with the composition
(see \cite[p. 67, \S $2.2.13$]{kuli}):
\begin{equation}\label{eq:mon}
     H^{2p}(\mc{X}_\infty, \mb{Z}) \xrightarrow[\sim]{i_{u+1}^*} H^{2p}(\mc{X}_{e(u)}, \mb{Z}) \xrightarrow[\sim]{(i_u^*)^{-1}} H^{2p}(\mc{X}_\infty, \mb{Z}).
\end{equation}
Denote by  $N:= \log(T^{\mr{uni}})$, where 
 $T^{\mr{uni}}$ is the unipotent part of the monodromy operator $T$.
If we need to specify the family $\pi$ we will sometimes use the notations   
$T^{(\pi)}$ and $N^{(\pi)}$ instead of $T$ and $N$, respectively. 
The \emph{limit Hodge filtration} is defined as (see \cite[Theorem $6.16$]{schvar})
\[F^k H^{2p}(\mc{X}_\infty, \mb{C}):= \lim_{\Ima(u) \to \infty} \mr{exp}(-uN) (i_u^*)^{-1}(F^k H^{2p}(\mc{X}_u, \mb{C})).\]
The \emph{limit weight filtration} is defined as the unique increasing filtration 
\begin{equation}\label{eq:wt-filt}
    0 \subset W_0H^{2p}(\mc{X}_\infty, \mb{Q}) \subset W_1H^{2p}(\mc{X}_\infty, \mb{Q}) \subset ... \subset W_{2p} H^{2p}(\mc{X}_\infty, \mb{Q})=H^{2p}(\mc{X}_\infty, \mb{Q})
\end{equation}
satisfying the following two conditions (see \cite[Lemma $6.4$]{schvar})
\begin{enumerate}
    \item $N(W_l H^{2p}(\mc{X}_\infty, \mb{Q})) \subset W_{l-2}(H^{2p}(\mc{X}_\infty, \mb{Q}))$ and \vspace{0.2cm}
    \item $N^l: \mr{Gr}^W_{p+l} H^{2p}(\mc{X}_\infty, \mb{Q}) \to \mr{Gr}^W_{p-l} H^{2p}(\mc{X}_\infty, \mb{Q})$ is an isomorphism, where 
    \[\mr{Gr}^W_l H^{2p}(\mc{X}_\infty, \mb{Q}):= (W_l H^{2p}(\mc{X}_\infty, \mb{Q}))/(W_{l-1} H^{2p}(\mc{X}_\infty, \mb{Q})).\]
\end{enumerate}

Steenbrink in \cite{ste1}
retrieved the limit weight filtration using a spectral sequence.
We recall the $E_1$-terms of the spectral sequence. For this we will need to recall the \emph{alternating restriction map}.
This is a linear combination of the usual restriction maps with some of them multiplied by $(-1)$. The purpose of this is 
to ensure that the composition of two differential maps $d \circ d$ of the $E_1$-terms is equal to zero. In particular, suppose that 
$\mc{X}_0$ is a reduced simple normal crossings divisor with irreducible components $X_1,...,X_N$. Then (see \cite[XI-$12$]{pet} for notations),
\begin{equation}\label{eq:xk}
    X(k):=\mbox{ normalization of } \coprod\limits_{i_1<i_2<...<i_k} X_{i_1} \cap X_{i_2} \cap ... \cap X_{i_k}\, \mbox{ and }
\end{equation}
$\delta_j:X(k) \hookrightarrow X(k-1)$ is induced by the natural inclusion of the intersection of $k$ irreducible components
$X_{i_1} \cap X_{i_2} \cap ... \cap X_{i_k}$ into the intersection of $k-1$ components obtained by dropping the $i_j$-th component  
$X_{i_1} \cap ... \cap X_{i_{j-1}} \cap X_{i_{j+1}} \cap ... \cap X_{i_k}$. The \emph{alternating restriction morphism} is defined as 
\begin{equation}\label{eq:alt-res}
    \rho: H^q(X(k-1),\mb{Q}) \to H^q(X(k), \mb{Q}),\, \mbox{ where } \rho = \sum_{j=1}^N (-1)^j \delta_j^*.
\end{equation}
Although $\rho$ depends on the ordering of the components $X_i$, the kernel of $\rho$ is independent of the ordering.
We now recall the Steenbrink spectral sequence. 
 
 \begin{thm}[{\cite[Corollary $11.23$]{pet}}]\label{thm:spec}
 Suppose that the central fiber $\mc{X}_0$ of $\pi$ be a simple normal crossings divisor. Let  
 $X(k)$ be as above. Then, the spectral sequence 
  \[^{\infty} E_1^{p,q}:=\bigoplus\limits_{k \ge \max\{0,p\}}
   H^{q+2p-2k}(X(2k-p+1),\mb{Q})(p-k)  \]
   with the differential map $d:\, ^\infty E_1^{p-1,q} \to\, ^\infty E_1^{p,q}$ being a 
   combination of the alternating restriction morphism \eqref{eq:alt-res} and the Gysin morphism,
   degenerates at $E_2$. Moreover, $^{\infty} E_1^{p,q} \Rightarrow
H^{p+q}(\mc{X}_\infty, \mb{Q})$ with the weight filtration given by $^{\infty} E_2^{p,q} = 
\mr{Gr}^W_q H^{p+q}(\mc{X}_\infty, \mb{Q})$.
 \end{thm}

 \subsection{Specialization map}
  Suppose that the central fiber $\mc{X}_0$ is a reduced simple normal crossings divisor. 
  Using the Mayer-Vietoris 
  sequence observe that the weight filtration on 
  $H^{2p}(\mc{X}_0,\mb{Q})$ arises from the spectral sequence
  with $E_1$-terms:
  \[E_1^{p,q}=H^q(X(p+1),\mb{Q})\, \Rightarrow H^{p+q}(\mc{X}_0,\mb{Q})\]
  where the differential $d: E_1^{p-1,q} \to E_1^{p,q}$ is 
  the alternating restriction morphism \eqref{eq:alt-res} (see \cite[Example $3.5$]{ste1}).
  Note that, the spectral sequence
  degenerates at $E_2$.
  By the identification between $H^{2p}(\mc{X}_\infty,\mb{Z})$ and $H^{2p}(\mc{X}_s,\mb{Z})$ as in \eqref{eq:iden}, we  
  get a specialization morphism (see \cite[\S $2$]{indpre}) which is a morphism 
  of mixed Hodge structures:
  \[\mr{sp}: H^{2p}(\mc{X}_0,\mb{Z}) \to H^{2p}(\mc{X}_\infty,\mb{Z}),\]
  where $H^{2p}(\mc{X}_\infty,\mb{Q})$ is equipped with the 
  limit mixed Hodge structure. 
  
  \begin{rem}
   By the definition of $E_1^{p,q}$ and $^\infty E_1^{p,q}$ given above,
 we have a natural morphism from $E_1^{p,q}$ to 
 $^\infty E_1^{p,q}$, which commutes with the 
 respective differential maps $d$. As a result, 
this induces a morphism of spectral sequences:
\begin{equation}\label{eq:phi}
 \phi: E_2^{p,q} \to\,  ^\infty E_2^{p,q}.
\end{equation}
  \end{rem}

  We now compute the kernel over the
  weight graded pieces of the specialization morphism. 

  \begin{lem}\label{lem:zero}
      Suppose that $\mc{X}_0$ is a simple normal crossings divisor. Then. the composition
       \[H^{q-2}(X(p+2),\mb{Q}) \xrightarrow{i_*} H^{q}(X(p+1),\mb{Q}) \xrightarrow{\rho}  H^{q}(X(p+2),\mb{Q})\, \mbox{ is the zero
  map}, \]
  where $X(m)$ is as in Theorem \ref{thm:spec}, $i_*$ is the Gysin map and $\rho$ is the alternating restriction morphism \eqref{eq:alt-res}. 
  \end{lem}

  \begin{proof}
    Let $Y$ be an irreducible component of $X(p+2)$. Since $\mc{X}_0$ is a simple normal crossings divisor,
    there exists exactly two irreducible components, say $Y_1, Y_2$ of $X(p+1)$ such that $Y$ is an irreducible 
    component of $Y_1 \cap Y_2$. Denote by $i_1: Y \hookrightarrow Y_1,\, i_2:Y \hookrightarrow Y_2$ the natural inclusions.
    It suffices to show that for any $\alpha \in H^{q-2}(Y,\mb{Q})$, 
    the image under the composition:
    \[H^{q-2}(Y,\mb{Q}) \xrightarrow{(i_{1,_*},i_{2,_*})} H^q(Y_1,\mb{Q}) \oplus H^q(Y_2,\mb{Q}) \xrightarrow{i_1^*-i_2^*} H^q(Y,\mb{Q}) \]
    is zero. This follows from the fact that, \[i_1^*i_{1,_*}\alpha=\alpha \cup c_1\left(\N_{Y|Y_1 \cup Y_2}\right)=i_2^*i_{2,_*}\alpha,\]
    where the two equalities follows from the fact that $\mc{X}_0$ is a simple normal crossings divisor.
    This proves the lemma.    
  \end{proof}
  
\begin{prop}\label{prop01}
 For $p \ge 0$, we have an exact sequence of the form:
 \[H^{q-2}(X(p+2),\mb{Q}) \to\, \,  E_2^{p,q} \xrightarrow{\phi}\, \,  ^\infty E_2^{p,q} \]
 where the first 
 morphism is induced by the Gysin morphism \[H^{q-2}(X(p+2),\mb{Q}) \to H^q(X(p+1),\mb{Q})=E_1^{p,q}\]
 and $\phi$ is as in \eqref{eq:phi}.
\end{prop}

\begin{proof}
 By Lemma \ref{lem:zero}, the composed morphism 
\begin{equation}\label{eq:prop01}
    H^{q-2}(X(p+2),\mb{Q}) \to H^{q}(X(p+1),\mb{Q}) \to  H^{q}(X(p+2),\mb{Q})\, \mbox{ is the zero
  map},
\end{equation}
where the first morphism is simply the Gysin morphism and the second morphism is the alternating restriction map.
Therefore, there is a natural map from 
$H^{q-2}(X(p+2),\mb{Q})$ to $E_2^{p,q}$. 
The difference between the spectral sequences $E_1^{p,q}$ and $^{\infty}E_1^{p,q}$ is that the 
differential map in the latter case also allows Gysin morphism. 
In particular, by Theorem \ref{thm:spec}, the kernel of the composition:
\[E_1^{p,q} = H^q(X(p+1),\mb{Q}) \to\,  ^{\infty} E_1^{p,q} \to\, \frac{^{\infty} E_1^{p,q}}{d_1:\, ^{\infty} E_1^{p-1,q} \to\, ^{\infty} E_1^{p,q}} \]
is isomorphic to the image of the Gysin map from $H^{q-2}(X(p+2),\mb{Q})$ to $H^q(X(p+1),\mb{Q})$. 
Using \eqref{eq:prop01}, it is easy to check that the kernel of the 
morphism $\phi$ is isomorphic to the image of the Gysin map. This 
proves the proposition.
\end{proof}

\subsection{Cycle class map and the cohomological singular Hodge conjecture}\label{sec:cyc}
The following result gives us a functorial cycle class map for arbitrary varieties that agree with the classical definition 
in the case of non-singular, projective varieties.
\begin{thm}{\cite[Corollary A.4]{soub}}\label{cor:cyc}
 There is a ``cycle class'' natural transformation of contravariant functors from the category of projective varieties over $\mb{C}$
 to the category of commutative, graded rings:
  \[ \bigoplus_{p} \mr{cl}^p : \bigoplus_{p}A^{p}(-) \to \bigoplus_{p} H^{2p}_{\mr{Hdg}}(-),\]
  where $H^{2p}_{\mr{Hdg}}(-)=F^p\mr{Gr}^W_{2p} H^{2p}(-,\mb{Q})= F^p\mr{Gr}^W_{2p} H^{2p}(-,\mb{C}) \cap \mr{Gr}^W_{2p} H^{2p}(-,\mb{Q})$.
\end{thm}

Let $X$ be a projective variety and $\dim(X) >p> \dim(X_{\mr{sing}})$ be an integer. 
  We say that $X$ \emph{satisfies} $\mr{SHC}(p)$  if 
   the cycle class map:
   \begin{equation}\label{eq:shc}
       \mr{cl}^p: A^p(X) \to H^{2p}_{\mr{Hdg}}(X)
   \end{equation}
  is surjective. 
 In the case when $X$ is non-singular and projective, this simply is  
 the classical Hodge conjecture (in weight $p$), which we will denote 
 by $\mr{HC}(p)$.

  \section{Mumford-Tate families}\label{sec:mt}
  Hodge theory of singular projective varieties is significantly more complicated than smooth, projective varieties, 
  due to the absence of standard results such as Hard Lefschetz, Poincar\'{e} duality.
  Moreover, singular projective varieties are often not equipped with a pure Hodge structure.
  In this section, we embed such singular varieties as the central fiber of a smooth family and study Hodge theoretic properties 
  inherited by the central fiber from the smooth ones. The difficulty of this approach lies in the fact that 
  Hodge lattices are mysterious objects and their deformations are difficult to predict. 
  For this reason we introduce the notion of Mumford-Tate families. Deformation of a Hodge lattice in such a family is well-behaved (see \S \ref{sec:mt-def}). There are several well-known examples of Mumford-Tate families. In Theorem \ref{thm:mt-exa}
  we show that the property of a family being Mumford-Tate is preserved under the usual operations of blow-ups, taking products, Hilbert schemes of points.
  This gives a recipe to produce new Mumford-Tate families from known ones (Theorem \ref{thm:lim-mt}).

  Consider a smooth family of projective varieties 
  \[\pi: \mc{X} \to \Delta^*.\]  
  We will use notations from \S \ref{se01}.
Denote by 
  \[H^{2p}_{\mr{Hdg}}(\mc{X}_\infty):= H^{p,p} \mr{Gr}^W_{2p} H^{2p}(\mc{X}_\infty, \mb{C}) \cap 
   \mr{Gr}^W_{2p} H^{2p}(\mc{X}_\infty, \mb{Q})  \]
 the vector space of \emph{limit Hodge classes}. 
 Since the logarithm $N$ of the monodromy operator induces a morphism of mixed Hodge structures, 
 it restricts to the vector space of Hodge classes:
 \[N: H^{2p}_{\mr{Hdg}}(\mc{X}_\infty) \to H^{2p}_{\mr{Hdg}}(\mc{X}_\infty).\]
 Denote by $H^{2p}_{\mr{Hdg}}(\mc{X}_\infty)^{\mr{inv}}$ the kernel of this morphism.
 These are the monodromy invariant Hodge classes in the limit mixed Hodge structure $H^{2p}(\mc{X}_\infty, \mb{Q})$.
 Since the limit Hodge filtration on $H^{2p}(\mc{X}_\infty, \mb{Q})$ arises as the limit of the 
 Hodge classes (see \S \ref{sec:lhf}), the limit of the Hodge classes are contained in $H^{2p}_{\mr{Hdg}}(\mc{X}_\infty)^{\mr{inv}}$.
 A natural question to ask is whether the inclusion is in fact an equality i.e.,

{\bf{Question}:} Does all the classes in $H^{2p}_{\mr{Hdg}}(\mc{X}_\infty)^{\mr{inv}}$ arise as the limit (as in \S \ref{sec:lhf}) of 
the elements in $H^{2p}_{\mr{Hdg}}(\mc{X}_t)$ as $\Ima(t)$ tends to $\infty$?

 The answer is no (see Example \ref{nonmtf}). However, if a family does give a positive answer to the question, then 
 the extension of the family to the entire disc $\Delta$ satisfies interesting Hodge theoretic properties. We will call 
 such families Mumford-Tate family (precise definition below). We will give numerous examples of Mumford-Tate families. 
 We will observe in the final section that under certain conditions, the central fiber of (an extension of) a Mumford-Tate family satisfies
 the singular Hodge conjecture.

\subsection{Mumford-Tate families}\label{sec:mt-def}
Denote by $H^{2p}_{\mr{Hdg}}(\mc{X}_t)^{\mr{inv}}$ the space of monodromy invariant classes in $H^{2p}_{\mr{Hdg}}(\mc{X}_t)$. 
Note that, any monodromy invariant Hodge class on a general fiber of $\pi$ extends to a (monodromy invariant) Hodge class on $\mc{X}_\infty$.
In particular, for a general $t \in \mf{h}$, the natural morphism 
\[H^{2p}_{\mr{Hdg}}(\mc{X}_t)^{\mr{inv}} \hookrightarrow H^{2p}(\mc{X}_t, \mb{Q}) \xrightarrow[(i_t^*)^{-1}]{\sim} H^{2p}(\mc{X}_\infty, \mb{Q})\]
factors through $H^{2p}_{\mr{Hdg}}(\mc{X}_\infty)^{\mr{inv}}$. We will say that $\pi$ is a \emph{Mumford-Tate family of weight} $p$, if the morphism 
from $H^{2p}_{\mr{Hdg}}(\mc{X}_t)^{\mr{inv}}$ to $H^{2p}_{\mr{Hdg}}(\mc{X}_\infty)^{\mr{inv}}$ is surjective.

\begin{exa}
Let $\pi$ be a trivial family i.e., $\mc{X}_{\Delta^*}$ is isomorphic to $X \times \Delta^*$ for some 
smooth, projective variety $X$. This implies that there is no non-trivial monodromy action on $H^{2p}(\mc{X}_t, \mb{Q})$ Therefore, 
\[H^{2p}_{\mr{Hdg}}(\mc{X}_t, \mb{Q})^{\mr{inv}} \cong H^{2p}_{\mr{Hdg}}(\mc{X}_t, \mb{Q}),\,  \mr{Gr}^W_{2p} H^{2p}(\mc{X}_\infty, \mb{Q}) \cong H^{2p}(\mc{X}_\infty, \mb{Q})\]
and $i_t^*$ is an isomorphism of pure Hodge structures. This implies the image of 
 \[ H^{2p}_{\mr{Hdg}}(\mc{X}_t, \mb{Q})^{\mr{inv}} \hookrightarrow H^{2p}(\mc{X}_t, \mb{Q}) \stackrel{i_t^*}{\cong}
    H^{2p}(\mc{X}_\infty, \mb{Q}) \cong \mr{Gr}^W_{2p} H^{2p}(\mc{X}_\infty, \mb{Q})\]
 coincides with $H^{2p}_{\mr{Hdg}}(\mc{X}_\infty, \mb{Q})^{\mr{inv}}$. Therefore, $\pi$ is a Mumford-Tate family.
\end{exa}

\subsection{Hard Lefschetz and Mumford-Tate families}
Mumford-Tate families satisfy a symmetric relation induced by the hard Lefschetz theorem (see Proposition \ref{prop:hl}).

Suppose that the relative dimension of the smooth family $\pi$ is $n$.
Denote by $\mb{L}$ the relative hyperplane section on $\mc{X}$. 
Fix an integer $p<n$ such that $n-p$ is even. By the hard Lefschetz theorem, we have an isomorphism 
\begin{equation}\label{eq:relhl}
     \cup \mb{L}^{^p}: H^{n-p}(\mc{X}_\infty, \mb{Q}) \xrightarrow{\sim} H^{n+p}(\mc{X}_\infty, \mb{Q})
\end{equation}
given by the cup-product with $p$-copies of the relative hyperplane section $\mb{L}$. We prove:

\begin{prop}\label{prop:hl}
     The family $\pi$ is a Mumford-Tate family of weight $n-p$ if and only if it is Mumford-Tate of weight $n+p$.
\end{prop}

\begin{proof}
Suppose first that $\pi$ is a Mumford-Tate family of weight $n-p$. 
By Corollary \ref{cor:mon}, we have for any $\alpha \in H^{n-p}(\mc{X}_\infty, \mb{Q})$:
\[N(\mb{L}^{^p} \cup \alpha)=\mb{L}^{^p} \cup N(\alpha), \]
i.e., $N$ commutes with the isomorphism \eqref{eq:relhl}. By the unicity of the weight filtration 
satisfying the two conditions in \S \ref{sec:lhf}, we conclude that \eqref{eq:relhl} is a morphism of limit 
mixed Hodge structures. Therefore, 
\[H^{n-p}_{\mr{Hdg}}(\mc{X}_\infty, \mb{Q})^{\mr{inv}} \cong H^{n+p}_{\mr{Hdg}}(\mc{X}_\infty, \mb{Q})^{\mr{inv}}.\]
Then, by the hard Lefschetz theorem we have for general $t \in \Delta^*$:
\[H^{n+p}_{\mr{Hdg}}(\mc{X}_t,\mb{Q})^{\mr{inv}} \cong H^{n-p}_{\mr{Hdg}}(\mc{X}_t, \mb{Q})^{\mr{inv}} \xrightarrow{\sim} H^{n-p}_{\mr{Hdg}}(\mc{X}_\infty, \mb{Q})^{\mr{inv}} \cong H^{n+p}_{\mr{Hdg}}(\mc{X}_\infty, \mb{Q})\]
    where the second isomorphism follows from the assumption that $\pi$ is Mumford-Tate of weight 
    $n-p$. This implies that $\pi$ is Mumford-Tate of weight $n+p$. The converse follows identically. This proves the proposition.
\end{proof}

\subsection{Relative cycle class}\label{sec:relcc}
Given a smooth, projective variety, the cohomology class of an algebraic cycle is known to be a Hodge class.
It is unknown whether an analogous result holds in a relative setup. In particular, given a family $\pi$ as in \S \ref{se01},
we ask: Is the the cohomology class of a relative cycle is a limit Hodge class? We answer this in Propositions \ref{prop02} and \ref{prop04}.

Let $\mc{Z} \subset \mc{X}_{\Delta^*}$ be a closed 
subscheme of $\mc{X}_{\Delta^*}$, flat over $\Delta^*$
and of relative dimension $n-p$.
The fundamental class  
 of $\mc{Z}$ defines a global section $\gamma_{_{\mc{Z}}}$ of the local system $\mb{H}^{2p}:= R^{2p}\pi_*\mb{Z}$ 
 such that for every $t \in \Delta^*$, 
 the value $\gamma_{_{\mc{Z}}}(t) \in H^{2p}(\mc{X}_t,\mb{Z})$
 of $\gamma_{_{\mc{Z}}}$ at the point $t$
 is simply the fundamental class of $\mc{Z}_t:=\mc{Z} \cap \mc{X}_t$ in $\mc{X}_t$ (see \cite[\S $19.2$]{fult} and 
 \cite[\S B.$2.9$]{pet} for details). The pull-back of the local system $\mb{H}^{2p}$
under the exponential map $e:\mf{h} \to \Delta^*$ is a 
trivial local system with fiber $H^{2p}(\mc{X}_\infty, \mb{Z})$. The global section $\gamma_{_{\mc{Z}}}$ defines an element of $H^{2p}(\mc{X}_\infty, \mb{Z})$, which we again denote by $\gamma_{_{\mc{Z}}}$, such that for every $s \in \mf{h}$, the image $i_s^*(\gamma_{_{\mc{Z}}})$ is the fundamental class of $\mc{Z} \cap \mc{X}_{e(s)}$ in $\mc{X}_{e(s)}$, where $i_s$ is the natural inclusion of $\mc{X}_{e(s)}$ into $\mc{X}_\infty$.  
 
 \begin{defi}\label{defi:lim-alg}
 Denote by $H^{2p}_A(\mc{X}_\infty)$
the sub-vector space of $H^{2p}(\mc{X}_\infty, \mb{Q})$ 
generated by all such elements of the form $\gamma_{_{\mc{Z}}}$
arising from a $\Delta^*$-flat closed subscheme of 
relative dimension $n-p$ in $\mc{X}_{\Delta^*}$. We call 
$H^{2p}_A(\mc{X}_\infty)$ the \emph{limit 
algebraic cohomology group}.
 \end{defi}

  \begin{prop}\label{prop02}
  The limit algebraic cohomology group is contained in the monodromy invariant part of the 
  limit Hodge cohomology group i.e., the natural inclusion 
  $H^{2p}_A(\mc{X}_\infty) \subset H^{2p}(\mc{X}_\infty,
  \mb{Q})$ factors through $H^{2p}_{\mr{Hdg}}(\mc{X}_\infty)^{\mr{inv}}$.
  \end{prop}

  \begin{proof}
   Take $\gamma \in H^{2p}_A(\mc{X}_\infty)$. 
   By construction, there exist 
   $\Delta^*$-flat closed subschemes $\mc{Z}_1,...,\mc{Z}_r$
   of relative dimension $n-p$ in $\mc{X}_{\Delta^*}$
   such that $\gamma=\sum a_i\gamma_{_{\mc{Z}_i}}$ for $a_i \in \mb{Q}$ and $\gamma_{_{\mc{Z}_i}} \in H^{2p}_A(\mc{X}_\infty)$
   is as defined above, arising from the fundamental
   class of $\mc{Z}_i$. 
   By construction, each $\gamma_{_{\mc{Z}_i}}$ arises from a 
   global section of the local system $\mb{H}^{2p}$.
   Hence, $\gamma_{_{\mc{Z}_i}}$ is monodromy invariant i.e.,
   $T(\gamma_{_{\mc{Z}_i}})=\gamma_{_{\mc{Z}_i}}$ for $1 \le i \le r$.
   This implies $N \gamma_{_{\mc{Z}_i}}=0$ for $1 \le i \le r$.
   
   As the cohomology class of $\mc{Z}_i \cap \mc{X}_{e(s)}$
   lies in $F^pH^{2p}(\mc{X}_{e(s)}, \mb{Q})$, we have 
   $\gamma_{_{\mc{Z}_i}} \in F^p_s H^{2p}(\mc{X}_\infty, \mb{Q})$
   for all $s \in \mf{h}$ (notations as in \S \ref{sec:lhf}).
   This implies $\gamma_{_{\mc{Z}_i}}$
   lies in $\mr{exp}(2\pi i sN)F^p_sH^{2p}(\mc{X}_\infty, \mb{Q})$ for every $s \in \mf{h}$.
   Recall from \S \ref{sec:lhf} that $F^pH^{2p}(\mc{X}_\infty, \mb{Q})$ contains 
   the limit of $\mr{exp}(2\pi i sN)F^p_sH^{2p}(\mc{X}_\infty, \mb{C})$ as $\mr{Im}(s)$ approaches $\infty$. Hence,
   $\gamma_{_{\mc{Z}_i}} \in F^pH^{2p}(\mc{X}_\infty, \mb{Q})$.
   As $\gamma_{_{\mc{Z}_i}}$ is monodromy invariant
   and a rational class,
   it must lie in $W_{2p} H^{2p}(\mc{X}_\infty, \mb{Q})$ (use the 
   invariant cycle theorem along with the fact that the degree $2p$ 
   cohomology of the central fiber is of weight 
   at most $2p$). 
   Therefore, $\gamma \in H^{2p}_{\mr{Hdg}}(\mc{X}_\infty)^{\mr{inv}}$.
   This proves the first part of the proposition.
  \end{proof}

  We now ask when is $H^{2p}_A(\mc{X}_\infty)$
  isomorphic to $H^{2p}_{\mr{Hdg}}(\mc{X}_\infty)^{\mr{inv}}$?
  One can naively guess that if the general fibers in the family $\pi$ satisfy the 
  Hodge conjecture then this happens. However, this is not enough (see \S \ref{nonmtf} below).  
  In particular, one needs to assume that the family $\pi$ is Mumford-Tate. In particular,

\begin{prop}\label{prop04}
  Suppose that $\pi$ is a Mumford-Tate family of weight $p$. 
  If a general fiber of $\pi$ satisfies HC$(p)$,
   then the inclusion from $H^{2p}_A(\mc{X}_\infty)$ 
   to $H^{2p}_{\mr{Hdg}}(\mc{X}_\infty)^{\mr{inv}}$ is an isomorphism.
\end{prop}

 Note that, by \emph{general} in the statement of the proposition,
we mean the complement of \emph{finitely many} proper, 
closed subvarieties 
of the punctured disc $\Delta^*$.

\begin{proof}
   Since $\pi$ is a Mumford-Tate family, for a general $t \in \mf{h}$ the morphism 
   \[H^{2p}_{\mr{Hdg}}(\mc{X}_t)^{\mr{inv}} \to H^{2p}_{\mr{Hdg}}(\mc{X}_\infty)^{\mr{inv}}\]
   mentioned in \S \ref{sec:mt-def} is an isomorphism. By assumption, $\mc{X}_t$ satisfies $\mr{HC}(p)$.
   In particular, the natural inclusion from $H^{2p}_A(\mc{X}_t)^{\mr{inv}}$ to $H^{2p}_{\mr{Hdg}}(\mc{X}_t)^{\mr{inv}}$
   is an isomorphism for general $t \in \mf{h}$. Moreover, the natural morphism 
   \[i_t^*: H^{2p}_A(\mc{X}_\infty) \to H^{2p}_A(\mc{X}_t)^{\mr{inv}}\]
   is an isomorphism for $t \in \mf{h}$ general. Indeed, injectivity is obvious. To prove surjectivity, 
   consider an algebraic cycle $\mc{Z}_t$ of codimension $p$ in $\mc{X}_t$. Since $t \in \mf{h}$ is general, $\mc{Z}_t$ extends to 
   a $\Delta^*$-flat closed subscheme $\mc{Z} \subset \mc{X}$ of relative dimension $n-p$ such that the fiber over $t$ is $\mc{Z}_t$. 
   In some sense, $\mc{Z}$ is the Zariski closure of    $\mc{Z}_t$ in $\mc{X}$. Then, the cohomology class of $\mc{Z}$ defines an 
   element $\gamma \in H^{2p}_A(\mc{X}_\infty)$ such that $i_t^*(\alpha)=[\mc{Z}_t]$. This implies surjectivity of $i_t^*$ above.
   Therefore, we have the following composed morphism which is an isomorphism:
   \[H^{2p}_A(\mc{X}_\infty) \xrightarrow[\sim]{i_t^*} H^{2p}_A(\mc{X}_t)^{\mr{inv}} \cong H^{2p}_{\mr{Hdg}}(\mc{X}_t)^{\mr{inv}} \cong H^{2p}_{\mr{Hdg}}(\mc{X}_\infty)^{\mr{inv}}.\]
   This proves the proposition.   
\end{proof}

\subsection{Non-example of Mumford-Tate families}\label{nonmtf}
 Recall for $d \ge 4$, the Noether-Lefschetz theorem states that a very general smooth, degree $d$ 
 surface in $\p3$ has Picard number $1$. The Noether-Lefschetz  
 locus parametrizes smooth degree $d$ surfaces in 
 $\p3$ with Picard number at least $2$. See  \cite{Dcont, D3, dan7} for some its geometric 
 properties. This means that 
 there are smooth families $\pi: \mc{X} \to \Delta$
 of hypersurfaces in $\p3$ such that $0 \in \Delta$ lies on the Noether-Lefschetz locus and 
 $\Delta^*$ does not intersect the Noether-Lefschetz locus. 
 Since $\pi$ is a smooth family, the local system $R^2\pi_*\mb{Q}$ does not have any monodromy over 
 the punctured disc. 
 Then, $H^2(\mc{X}_\infty, \mb{Q}) \cong H^2(\mc{X}_0.\mb{Q})$ as pure Hodge structures. In particular, 
 by the condition on the central fiber $\mc{X}_0$, 
 the rank of the Hodge lattice in $H^2(\mc{X}_\infty,\mb{Q})$ is at least $2$. But the rank of the 
 Hodge lattice in $H^2(\mc{X}_s,\mb{Q})$ is $1$ for any $s \in \Delta^*$. Since the pullback morphism 
 $i_s^*$ is an isomorphism, this implies that there is a Hodge class on $H^2(\mc{X}_\infty,\mb{Q})$ that 
 does not pullback to a Hodge class on $H^2(\mc{X}_s,\mb{Q})$. Hence, $\pi$ cannot be a Mumford-Tate family.

\subsection{Examples of Mumford-Tate families} 

\begin{thm}\label{thm:mt-exa}
Let $W$ be a smooth, projective variety of dimension $n$ and 
     \[\pi_1: \mc{X} \to \Delta^*, \, \pi_2: \mc{Y} \to \Delta^*\]
    be two smooth, projective families of ample hypersurfaces in $W$ in the sense that for all $t \in \Delta^*$, $\mc{X}_t$ and $\mc{Y}_t$
    are ample hypersurfaces in $W$. Then, the following holds true: 
    \begin{enumerate}
        \item If $n$ is even then $\pi_1, \pi_2$ are Mumford-Tate families.
        \item If $\pi_1$ and $\pi_2$ are Mumford-Tate families, then $\pi_1 \times \pi_2$ are Mumford-Tate families.
        \item If $\pi_1$ is a Mumford-Tate family, then the associated family of Hilbert scheme of points of length $2$ is a Mumford-Tate family.
        \item Let $Z \subset \mc{X}$ be a section of $\pi_1$ i.e., the composition $\pi_1:Z \to \mc{X} \to \Delta^*$ is an isomorphism.
        If $\mc{X}$ is a Mumford-Tate family, then the blow-up pf $\mc{X}$ along $Z$ is a Mumford-Tate family.
        \end{enumerate}
\end{thm}

\begin{proof}
    $(1)$ By the Lefschetz hyperplane section theorem, for any $t \in \Delta^*$ and $p \ge 0$, we have 
    $H^{2p}(W,\mb{Q}) \cong H^{2p}(\mc{X}_t,\mb{Q})$. This implies $R^{2p}\left(\pi_1\right)_*\mb{Q}$ is a trivial local system, isomorphic 
    to $H^{2p}(W,\mb{Q}) \times \Delta^*$. Therefore, for all $t \in \Delta^*$, 
    \[H^{2p}_{\mr{Hdg}}(\mc{X}_t,\mb{Q})^{\mr{inv}} \cong H^{2p}_{\mr{Hdg}}(W,\mb{Q}) \cong H^{2p}_{\mr{Hdg}}(\mc{X}_\infty)^{\mr{inv}}. \]
    Hence, $\pi_1$ is a Mumford-Tate family. Similarly, $\pi_2$ is a Mumford-Tate family. This proves $(1)$.

    $(2)$ The family $\pi_1 \times \pi_2$ is of relative dimension $2n-2$.
    By Proposition \ref{prop:hl}, it suffices to check that $\pi_1 \times \pi_2$ 
    is a Mumford-Tate family of weight $p$ for all $p \le n-1$.
 The K\"{u}nneth decomposition induces the isomorphism:
    \begin{equation}\label{eq:relkun}
     \bigoplus_{i=0}^{2p} \left(H^{i}(\mc{X}_\infty, \mb{Q}) \otimes H^{2p-i}(\mc{Y}_\infty, \mb{Q})\right)  \xrightarrow{\sim} H^{2p}(\mc{X}_\infty \times_{\mf{h}} \mc{Y}_\infty, \mb{Q}).
    \end{equation}
 By the Lefschetz hyperplane theorem, for $i< (n-1)$ and $t \in \Delta^*$,
 we have $H^i(\mc{Y}_t, \mb{Q}) \cong H^i(W,\mb{Q}) \cong H^i(\mc{X}_t, \mb{Q})$. 
 This implies for $i<n-1$ we have an isomorphism of pure Hodge structures: 
 \[(H^i(\mc{X}_t,\mb{Q}) \otimes H^{2p-i}(\mc{Y}_t,\mb{Q}))^{\mr{inv}} \cong H^i(W,\mb{Q}) \otimes H^{2p-i}(W,\mb{Q}) \cong \left(H^i(\mc{X}_\infty,\mb{Q})
 \otimes H^{2p-i}(\mc{Y}_\infty, \mb{Q})\right)^{\mr{inv}}.\]
 Hence, the Hodge classes on the left-hand side maps isomorphically to that on the right hand side.
 Moreover, since $\pi_1, \pi_2$ are Mumford-Tate families, if $2p=n-1$, then 
 \[H^{2p}_{\mr{Hdg}}(\mc{X}_t, \mb{Q})^{\mr{inv}} \cong H^{2p}_{\mr{Hdg}}(\mc{X}_\infty, \mb{Q})^{\mr{inv}}\, \mbox{ and }
 H^{2p}_{\mr{Hdg}}(\mc{Y}_t, \mb{Q})^{\mr{inv}} \cong H^{2p}_{\mr{Hdg}}(\mc{Y}_\infty, \mb{Q})^{\mr{inv}}, \mbox{ for } t \in \Delta^*\, \mbox{ general}.\]
 This implies for $t \in \Delta^*$ general and $2p \le n-1$, that $H^{2p}_{\mr{Hdg}}(\mc{X}_t \times \mc{Y}_t,\mb{Q})^{\mr{inv}}$ is isomorphic to 
\[\bigoplus_{i=0}^{2p} \left(H^i(\mc{X}_t,\mb{Q}) \otimes H^{2p-i}(\mc{Y}_t, \mb{Q})\right)_{\mr{Hdg}}^{\mr{inv}} \cong \bigoplus_{i=0}^{2p} \left(H^i(\mc{X}_\infty,\mb{Q}) \otimes H^{2p-i}(\mc{Y}_\infty, \mb{Q})\right)_{\mr{Hdg}}^{\mr{inv}} \cong H^{2p}_{\mr{Hdg}}(\mc{X}_\infty \times_{\mf{h}} \mc{Y}_\infty)^{\mr{inv}}.\]
Hence, $\pi_1 \times \pi_2$ is a Mumford-Tate family. This proves $(2)$. 

$(3)$ Associated to this family $\pi$, we have the families of 
 symmetric products, Hilbert schemes of points and exceptional divisors of blow-ups along the diagonal:
 \[\pi^{(2)}: \mc{X}^{(2)} \to \Delta^*,\, \pi^{[2]}: \mc{X}^{[2]} \to \Delta^*\, \mbox{ and }\, \pi_{_E}: \mc{E} \to \Delta^*\]
such that for every $t \in \Delta^*$, the fiber $\left(\pi^{[2]}\right)^{-1}(t)=\mc{X}_t^{[2]}$ is
the Hilbert scheme of points in $\mc{X}_t$ of length $2$, 
$\left(\pi^{(2)}\right)^{-1}(t)=\mc{X}_t^{(2)}$ symmetric product of $\mc{X}_t$ with itself
and $\pi_E^{-1}(t)=\mc{E}_t$ is the exceptional divisor for the blow-up of $\mc{X}_t^{(2)}$ along the diagonal.
Note that, $\mc{X}_t^{[2]}$ is the blow-up of $\mc{X}_t^{(2)}$ along the diagonal $\Delta_t$.
 Then, we have an exact sequence of the form:
  \begin{equation}\label{eq:sym}
     ... \to H^{2p-1}(\mc{E}_t) \to H^{2p}(\mc{X}_t^{(2)}) \to H^{2p}(\mc{X}_t^{[2]}) \oplus H^{2p}(\Delta_t) \to H^{2p}(\mc{E}_t) \to 
  H^{2p+1}(\mc{X}_t^{(2)}) \to ... 
  \end{equation}
  Since $\Delta_t$ is a regular subvariety in $\mc{X}_t \times \mc{X}_t$,
  $\mc{E}_t$ is a $\mb{P}^{n-1}$-bundle over $\Delta_t$.
  Then, the Leray-Hirsch spectral sequence implies that  
\begin{equation}\label{eq:leray}
    H^{2p}(\mc{E}_t, \mb{Z}) \cong \bigoplus_i H^i(\mb{P}^{n-1}, \mb{Z}) \otimes H^{2p-i}(\mc{X}_t, \mb{Z}) \cong 
  \bigoplus_{i=0}^p H^{2i}(\mb{P}^{n-1}, \mb{Z}) \otimes H^{2p-2i}(\mc{X}_t, \mb{Z})
\end{equation}
 Therefore, 
 \begin{equation}\label{eq:kun3}
  H^{2p}_{\mr{Hdg}}(\mc{E}_\infty)^{\mr{inv}} \cong \bigoplus_{i=0}^p H^{2i}_{\mr{Hdg}}(\mb{P}^{n-1} \times \mf{h})^{\mr{inv}} \otimes H^{2p-2i}_{\mr{Hdg}}(\mc{X}_\infty)^{\mr{inv}} \cong \bigoplus_{i=0}^p H^{2i}_{\mr{Hdg}}(\mb{P}^{n-1})^{\mr{inv}} \otimes H^{2p-2i}_{\mr{Hdg}}(\mc{X}_t)^{\mr{inv}}   
 \end{equation}
which is isomorphic to $H^{2p}(\mc{E}_t)^{\mr{inv}}$. 
 Moreover, the Hodge classes of $H^{2p}(\mc{X}_\infty^{(2)}, \mb{Q})= H^{2p}(\mc{X}_\infty \times_{\mf{h}} \mc{X}_\infty, \mb{Q})^{\Sigma_2}$
  consists of the $\Sigma_2$-invariant Hodge classes in  $H^{2p}(\mc{X}_\infty \times_{\mf{h}} \mc{X}_\infty, \mb{Q})$, where $\Sigma_2$ is the 
  symmetric group over $2$ elements. 
  Since $\pi_1 \times \pi_1$ is a Mumford-Tate family (by $(2)$), 
  we have 
  \begin{equation}\label{eq:kun1}
      (H^{2p}_{\mr{Hdg}}(\mc{X}_{\infty} \times_{\mf{h}} \mc{X}_\infty,\mb{Q})^{\Sigma_2})^{\mr{inv}} \cong (H^{2p}_{\mr{Hdg}}(\mc{X}_t \times \mc{X}_t,\mb{Q})^{\Sigma_2})^{\mr{inv}}.
  \end{equation}
Since pullback commutes with monodromy, the exact sequence \eqref{eq:sym} induces the following 
commutative diagram of pure Hodge structures:
  \[\begin{diagram} 
     & &(H^{2p}_{\mr{Hdg}}(\mc{X}_{\infty} \times_{\mf{h}} \mc{X}_\infty,\mb{Q})^{\Sigma_2})^{\mr{inv}}&\rTo&H^{2p}_{\mr{Hdg}}(\mc{X}_{\infty}^{[2]})^{\mr{inv}} \oplus H^{2p}_{\mr{Hdg}}(\Delta_{\mc{X}_{\infty}})^{\mr{inv}}&\rTo&H^{2p}_{\mr{Hdg}}(\mc{E}_{\infty})^{\mr{inv}}\\
     & &\dTo^{\cong}&\circlearrowleft&\dTo&\circlearrowleft&\dTo^{\cong}\\
        0&\rTo&H^{2p}_{\mr{Hdg}}(\mc{X}_t^{(2)})^{\mr{inv}}&\rTo&H^{2p}_{\mr{Hdg}}(\mc{X}_t^{[2]})^{\mr{inv}} \oplus H^{2p}_{\mr{Hdg}}(\Delta_{\mc{X}_t})^{\mr{inv}}&\rTo&H^{2p}_{\mr{Hdg}}(\mc{E}_t)^{\mr{inv}}
    \end{diagram}\]
 where the first (resp. last) isomorphism follows from \eqref{eq:kun1} (resp. \eqref{eq:kun3}).
By construction the middle vertical arrow is surjective (every Hodge class on a general fibers extends to a Hodge class on the entire family).
By diagram chase, observe that the middle vertical arrow is also injective, hence an isomorphism.  Therefore, $\pi^{[2]}$ is a Mumford-Tate family.
This proves $(3)$. 

$(4)$ Denote by $\pi:\mr{Bl}_Z\mc{X} \to \Delta^*$ the blow-up $\wt{\mc{X}}$ of $\mc{X}$ along $Z$. Then, for every $t \in \Delta^*$, the fiber $\pi^{-1}(t)$
is the blow-up of $\mc{X}_t$ at the point $Z_t$. Denote by $E_t$ the associated exceptional divisor. Since $Z_t$ is a point, $E_t \cong \mb{P}^{n-2}$.
Hence, 
\[H^{2p}(\wt{\mc{X}}_t,\mb{Q}) \cong H^{2p}(\mc{X}_t,\mb{Q}) \oplus H^{2p}(\mb{P}^{n-2},\mb{Q}).\]
Since $\pi_1$ is a Mumford-Tate family and $H^{2p}(\mb{P}^{n-2},\mb{Q})$ is monodromy invariant, this implies for general $t \in \Delta^*$ that 
\[H^{2p}_{\mr{Hdg}}(\wt{\mc{X}}_t)^{\mr{inv}} \cong H^{2p}_{\mr{Hdg}}(\mc{X}_t)^{\mr{inv}} \oplus H^{2p}(\mb{P}^{n-2},\mb{Q}) \cong 
H^{2p}_{\mr{Hdg}}(\mc{X}_{\infty})^{\mr{inv}} \oplus H^{2p}(\mb{P}^{n-2},\mb{Q}) \cong H^{2p}_{\mr{Hdg}}(\mc{X}_\infty)^{\mr{inv}}. \]
Hence, $\pi$ is a Mumford-Tate family. This proves $(4)$ and hence the theorem. 
\end{proof}

  \subsection{Mumford-Tate families via correspondences}\label{sec:corr-mt}
  We can produce new Mumford-Tate families using a classical tool in Hodge theory known as correspondences.
  Correspondences are morphisms of pure Hodge structures induced by an algebraic cycle 
  (see \cite[\S $11.3.3$]{v4} for a detailed description). We will show that 
  varieties linked by correspondences inherit the property of Mumford-Tate (see Theorem \ref{thm:lim-mt}, Examples \ref{sec:fano} and \ref{sec:mod}).
  
  Consider two smooth, projective families over the punctured disc;
  \[\pi_1: \mc{X} \to \Delta^*\, \mbox{ and } \pi_2: \mc{Y} \to \Delta^*.\]
  of relative dimension $n$ and $m$, respectively. 
   Fix a positive integer $c$. 
  Let $\mc{Z} \subset \mc{X} \times_{\Delta^*} \mc{Y}$ be a closed subscheme,
  flat over $\Delta^*$ and for every $t \in \Delta^*$, the fiber $\mc{Z}_t$ is of 
  codimension $c$ in $\mc{X}_t \times \mc{Y}_t$. 
 Then, the cohomology class $[\mc{Z}]$ of $\mc{Z}$ defines a global section in the local system
 \[R^{2c} (\pi_1 \times \pi_2)_* \mb{Z},\, \mbox{ where } (\pi_1 \times \pi_2): \mc{X} \times_{\Delta^*} \mc{Y} \to \Delta^*\]
  is the natural morphism. This local system admits a Kunneth decomposition:
  \[ R^{2c} (\pi_1 \times \pi_2)_* \mb{Z} \cong \bigoplus_i \left( R^i \pi_{1,_*} \mb{Z} \otimes R^{2c-i}   \pi_{2,_*} \mb{Z}\right).  \]
 Denote by $[\mc{Z}]_i$ the image of the section $[\mc{Z}]$ in 
 $R^i \pi_{1,_*} \mb{Z} \otimes R^{2c-i}   \pi_{2,_*} \mb{Z}$ under the above identification.
 Note that $R^i \pi_{1,_*} \mb{Z}$ is Poincar\'{e} dual to the local system 
 $R^{2n-i} \pi_{1,_*} \mb{Z}$ i.e.,
 \[ R^i \pi_{1,_*} \mb{Z} \cong \left( R^{2n-i} \pi_{1,_*} \mb{Z} \right)^{\vee}.\]
 This means, $R^i \pi_{1,_*} \mb{Z} \otimes R^{2c-i}   \pi_{2,_*} \mb{Z}$ is isomorphic to 
 \[ \left( R^{2n-i} \pi_{1,_*} \mb{Z} \right)^{\vee} \otimes R^{2c-i}   \pi_{2,_*} \mb{Z}
  \cong \mr{Hom}\left( R^{2n-i} \pi_{1,_*} \mb{Z}, R^{2c-i}   \pi_{2,_*} \mb{Z}\right).\]
  Since $[\mc{Z}]_i$ is a global section, it defines a monodromy invariant morphism, say $\Phi_i$, 
  of local systems from $R^{2n-i} \pi_{1,_*} \mb{Z}$ to  $R^{2c-i}   \pi_{2,_*} \mb{Z}$ i.e.,
  we have the following commutative diagram:
  \begin{equation}\label{diag:corr}
   \begin{diagram}
     R^{2n-i} \pi_{1,_*} \mb{Z}&\rTo^{\Phi_i}&R^{2c-i}   \pi_{2,_*} \mb{Z}\\
     \dTo^{T^{(\pi_1)}}&\circlearrowleft&\dTo^{T^{(\pi_2)}}\\
     R^{2n-i} \pi_{1,_*} \mb{Z}&\rTo^{\Phi_i}&R^{2c-i}   \pi_{2,_*} \mb{Z}
         \end{diagram}
  \end{equation}
where $T^{(\pi_1)}$ (resp. $T^{(\pi_2)}$) are the monodromy operators acting on the local systems 
$R^{2n-i} \pi_{1,_*} \mb{Z}$ (resp. $R^{2c-i}   \pi_{2,_*} \mb{Z}$).
 Pulling back the local systems $R^{2n-i} \pi_{1,_*} \mb{Z}$ and $R^{2c-i} \pi_{2,_*} \mb{Z}$
   from $\Delta^*$ to the upper half plane $\mf{h}$ and then taking global sections, we get a morphism of cohomology groups
   \begin{equation}\label{eq:lim-corr}
   \Phi_i: H^{2n-i}(\mc{X}_\infty, \mb{Z}) \to H^{2c-i}(\mc{Y}_\infty, \mb{Q})\, \mbox{ which restricts to } \Phi_{i,t}: H^{2n-i}(\mc{X}_t,\mb{Z}) \to H^{2c-i}(\mc{Y}_t,\mb{Q})
    \end{equation}
    for all $t \in \mf{h}$.

\begin{defi}\label{def:corr}
 We will call such a morphism $\Phi_i$
 from $H^{2n-i}(\mc{X}_\infty, \mb{Z})$ to $H^{2c-i}(\mc{Y}_\infty, \mb{Q})$
   to be a \emph{correspondence between the families $\pi_1$ and $\pi_2$ induced 
 by $\mc{Z}$}. 
\end{defi}
 
  Although the morphism $\Phi_{i,t}$ is a morphism of pure Hodge structures, 
  $\Phi_i$ is not a morphism of pure Hodge structures. This is because 
  the cohomology of $\mc{X}_\infty$ and $\mc{Y}_\infty$ are equipped with limit 
  mixed Hodge structure induced by the  monodromy operator. Surprisingly, $\Phi_i$ is a morphism of 
  limit mixed Hodge structures.

  \begin{lem}\label{lem:lim}
   The morphism $\Phi_i$ defined in \eqref{eq:lim-corr} from 
   $H^{2n-i}(\mc{X}_\infty, \mb{Q})$ to $H^{2c-i}(\mc{Y}_\infty, \mb{Q})$ is a morphism 
   of mixed Hodge structures, when they are both equipped 
   with the limit mixed Hodge structures.
  \end{lem}

  \begin{proof}
     The commutative diagram \eqref{diag:corr} implies that monodromy commutes with the morphism 
   $\Phi_i$ i.e., $\Phi_i \circ T^{(\pi_1)} = T^{(\pi_2)} \circ \Phi_i$, where 
   $T^{(\pi_1)}$ (resp. $T^{(\pi_2)}$) is the monodromy operator associated to the family 
   $\pi_1$ (resp. $\pi_2$). Denote by 
   $N^{(\pi_1)}:= \log(T^{(\pi_1)})$ and $N^{(\pi_2)}:= \log(T^{(\pi_2)})$. The commutativity with 
   the monodromy operators induces the equality:
   \begin{equation}\label{eq:comm-mon}
    \Phi_i \circ N^{(\pi_1)} = N^{(\pi_2)} \circ \Phi_i.
   \end{equation}
     Recall from \cite[p. $106$]{morri}, the weight filtration on $H^k(\mc{X}_\infty, \mb{Q})$
   is defined inductively as follows: \[W_0H^k(\mc{X}_\infty):= \mr{Im}((N^{(\pi_1)})^k),\, \,  
   W_{2k-1} H^k(\mc{X}_\infty):= \ker((N^{(\pi_1)})^k).\]
   If for some fixed $l<k$, we have already 
   defined 
   \[ 0 \subset W_{l-1} \subset W_{2k-l} \subset W_{2k}=H^{2k}(\mc{X}_\infty) \]
   then define 
   \[W_l:= \mr{Im}((N^{(\pi_1)})^{k-l}|_{W_{2k-l}})+W_{l-1}\, \mbox{ and } W_{2m-l-1}:= \ker((N^{(\pi_1)})^{k-l}|_{W_{2k-l}}) + W_{l-1}.\]
   The weight filtration on $H^k(\mc{Y}_\infty, \mb{Q})$ is defined similarly with 
   $N^{(\pi_1)}$ substituted by $N^{(\pi_2)}$. Using the equality \eqref{eq:comm-mon} we conclude that the (limit) weight filtration is preserved under $\Phi_i$.
   
   Recall from \S \ref{sec:lhf} that the limit Hodge filtration is defined as:
   \[F^k H^{2n-i}(\mc{X}_\infty, \mb{C}):=
 \lim_{\Ima(s) \to \infty} \mr{exp}(2\pi \sqrt{-1} sN^{(\pi_1)})F^k_s 
 H^{2n-i}(\mc{X}_\infty, \mb{C}),\]
 where $F^k_s$ is the Hodge filtration induced by that on $H^{2n-i}(\mc{X}_s, \mb{C})$.
 The Hodge filtration on $H^{2c-i}(\mc{Y}_\infty)$ is defined similarly with $N^{(\pi_1)}$ replaced
 by $N^{(\pi_2)}$.
 By \cite[Lemma $11.41$]{v4}, the restriction of the morphism $\Phi$ to $\mc{X}_s$ induces a morphism of 
 pure Hodge structures from $H^{2n-i}(\mc{X}_s, \mb{C})$ to $H^{2c-i}(\mc{Y}_s, \mb{C})$.
 Finally, using the equality \eqref{eq:comm-mon} we conclude that the Hodge filtration is 
 preserved under $\Phi_i$. This proves the lemma.
  \end{proof}

  Using this we can now produce new Mumford-Tate families.

\begin{thm}\label{thm:lim-mt}
 Fix integers $c \ge 0$ and $p \ge 0$. Consider a smooth, projective family 
  \[\pi: \mc{X} \to \Delta^*\,  \mbox{ of relative dimension } n.\]
  Suppose there exists a finite number of smooth, projective families:
  \begin{equation}\label{eq:fam}
    \pi_i: \mc{Y}_i \to \Delta^*\, \mbox{ of relative dimension } m_i
  \end{equation}
  and closed subschemes $\mc{Z}_i \subset \mc{X} \times_{\Delta^*} \mc{Y}_i$ 
  which are flat over $\Delta^*$ and for every $t \in \Delta^*$, the fiber 
  $\mc{Z}_{i,t}$ over $t$ is of codimension $c$ in $\mc{X}_t \times \mc{Y}_{i,t}$
 such that the induced morphism as in \eqref{eq:lim-corr}:
 \begin{equation}\label{eq:fam2}
     \bigoplus_i \Phi_{p,t}^{(i)}: \bigoplus_i H^{2m_i-2p}(\mc{Y}_{i, t}, \mb{Q}) \to H^{2c-2p}(\mc{X}_t, \mb{Q})
 \end{equation}
  is surjective, for general $t \in \Delta^*$.
  If the families $\pi_i$ in \eqref{eq:fam} are Mumford-Tate of 
  weight $2m_i-2p$, then $\pi$ is a Mumford-Tate family of weight $2c-2p$.
  \end{thm}

  Before we prove the theorem we wanted remark that not all $\mc{Y}_i$'s need to be distinct.

  \begin{proof}
  Consider the following commutative diagram:
 \[\begin{diagram}
      \bigoplus_i H^{2m_i-2p}(\mc{Y}_{i, \infty}, \mb{Q}) &\rTo^{\bigoplus_i \Phi_p^{(i)}}&H^{2c-2p}(\mc{X}_\infty, \mb{Q})\\
     \dTo^{j_t^*}&\circlearrowleft&\dTo_{(j'_t)^*}\\
        \bigoplus_i H^{2m_i-2p}(\mc{Y}_{i,t}, \mb{Q})&\rTo^{\bigoplus_i \Phi_{p,t}^{(i)}}&H^{2c-2p}(\mc{X}_t, \mb{Q})    
    \end{diagram}\]
where $j_t$ (resp. $j'_t$) is the natural inclusion of $\mc{X}_t$ (resp. 
$\mc{Y}_t$) into $\mc{X}_\infty$ (resp. $\mc{Y}_\infty$). By Ehresmann's lemma, the two vertical arrows are isomorphisms.
By assumption, the bottom horizontal morphism is a surjective, morphism of pure Hodge structures.
This implies that the top horizontal map is a surjective morphism of mixed Hodge structures (see Lemma \ref{lem:lim}).
Hence, given $\gamma \in H^{2c-2p}_{\mr{Hdg}}(\mc{X}_\infty, \mb{Q})$, there exists 
\[\gamma' \in \bigoplus_i H^{2m_i-2p}_{\mr{Hdg}}(\mc{Y}_{i, \infty}, \mb{Q})\, \mbox{ such that } \bigoplus_i \Phi_p^{(i)}(\gamma')=\gamma.\]
Moreover by \eqref{diag:corr}, the morphism $\Phi_p^{(i)}$ commutes with the monodromy operator, for all $i$.
As a result, if $\gamma$ is monodromy invariant, $\gamma'$ can be assumed to be monodromy invariant.
Since $\pi_i$ are Mumford-Tate families for all $i$, \[\gamma'_t:= j_t^*(\gamma') \in \bigoplus_i H^{2m_i-2p}_{\mr{Hdg}}(\mc{Y}_{i,t}, \mb{Q})^{\mr{inv}}\]
for general $t \in \mf{h}$.
Since $\Phi_{p,t}^{(i)}$ are morphisms pure Hodge structures, $\oplus \Phi_{p,t}^{(i)}(\gamma'_t) \in H^{2c-2p}_{\mr{Hdg}}(\mc{X}_t,\mb{Q})^{\mr{inv}}$.
 In other words, $(j'_t)^*(\gamma) \in H^{2c-2p}_{\mr{Hdg}}(\mc{X}_t,\mb{Q})^{\mr{inv}}$.
Therefore, $\pi$ is a Mumford-Tate family. This proves the theorem.
  \end{proof}

\begin{exa}[Fano varieties of lines]\label{sec:fano}
We claim that families of Fano varieties of lines associated to families of cubic hypersurfaces in $\mb{P}^n$
are Mumford-Tate, if the former family is Mumford-Tate. More precisely, 
 let $\pi: \mc{X} \to \Delta^*$ be a smooth, projective family of cubic hypersurface in $\mb{P}^n$ and 
 $\pi_F:F(\mc{X}) \to \Delta^*$ the associated Fano variety of lines. In particular, $F(\mc{X})$
 is the universal family associated to a representable functor (see \cite[p. $79$]{huycub}). 
 For every $t \in \Delta^*$, $F(\mc{X})_t$ is the Fano variety of lines on $\mc{X}_t$. 
   The cohomology of $(F(\mc{X}))_t$ is controlled by the cohomology of $\mc{X}_s$ (see \cite[\S $2.4$]{huycub}). 
   Using \cite[Proposition $15$]{latv2} (see also \cite[Chapter $2$, \S $4$]{huycub}), 
   there is a surjective morphism from $H^{2p+2}(\mc{X}_t^{[2]},\mb{Q})$ to $H^{2p}(F(\mc{X}_t),\mb{Q})$
   induced by a correspondence, for all $t \in \Delta^*$ and $p>1$. By Theorem \ref{thm:mt-exa}, $\mc{X}^{[2]}$ is a Mumford-Tate family.
   Then, Theorem \ref{thm:lim-mt} implies that if $\pi$ is a Mumford-Tate family (for example, if $n$ is even), 
   then $F(\mc{X})$ is a Mumford-Tate family. 
\end{exa}

  \begin{exa} [Moduli  spaces of sheaves over surfaces]\label{sec:mod}
  Families of  moduli of sheaves associated to families of K3 surfaces or abelian surfaces are Mumford-Tate families,
  if the corresponding family of surfaces is a Mumford-Tate family. In particular, 
  let \[\pi: \mc{X} \to \Delta^*\] be a smooth family of K3 surfaces or abelian surfaces.
   Fix a positive integer $r>1$, relative Chern class of a line bundle on the family $c_1 \in \Gamma(R^2\pi_*\mb{Z})$
   and a global section $c_2 \in \Gamma(R^4 \pi_*\mb{Z})$. Suppose that there exists $t \in \Delta^*$ such that the for the 
   corresponding value $c_1(t) \in H^2(\mc{X}_t,\mb{Z})$, we have $\mr{gcd}(r,c_1.c_1(\mo_{\mc{X}_t}))=1$. 
   Since intersection product is preserved in flat families, this implies $\mr{gcd}(r,c_1.c_1(\mo_{\mc{X}_t}))=1$
   for all $t \in \Delta^*$. Then, using \cite[Theorems $4.6.5$ and $6.1.8$]{huy}, we conclude that there exists 
   a smooth, projective  family \[\pi(r,c_1,c_2): \mc{M}(r,c_1,c_2) \to \Delta^*\]
   such that for every $t \in \Delta^*$, the fiber over $t$ is the (non-singular) moduli space 
   of semi-stable sheaves on $\mc{X}_t$ with Chern character $r+c_1+c_2$. Moreover, there exists 
 a relative universal family $\mc{U}$ over $\mc{M}(r,c_1,c_2) \times_{\Delta^*} \mc{X}$.
 Note that, $H^1(\mc{X}_\infty, \mb{Z})=0=H^3(\mc{X}_\infty, \mb{Z})$. 
By \cite[Corollary $2$]{mark1}, 
 \[\bigoplus_{j=0}^2 \Phi_{2j,t}^{(i+j)}:  \bigoplus_{j=0}^2 H^{4-2j}(\mc{X}_t, \mb{Z}) \to H^{2i}(\mc{M}(r,c_1,c_2)_t, \mb{Z})\, 
 \mbox{ is surjective } \]
 where $\Phi_{2j,t}^{(i+j)}$ is the correspondence induced by the $i$-th Chern class $c_i(\mc{U}_t)$ of the universal bundle restricted to the 
 fiber over $t \in \Delta^*$. Then, 
Theorem \ref{thm:lim-mt} implies that if $\pi$ is a Mumford-Tate family, then so is the family $\pi(r,c_1,c_2)$ of moduli spaces of sheaves.
 \end{exa}

\section{Examples of (singular) Hodge conjecture}\label{sec:exa}
In this section we apply the above results to produce new examples of singular 
projective varieties satisfying the singular Hodge conjecture (Theorems \ref{th01} and \ref{thm:final}) and Jannsen's conjecture (Corollary \ref{cor:jann}).
Resolutions of these singular varieties often satisfy the Hodge conjecture (Theorem \ref{thm:final}). 

\subsection{SHC for simple normal crossings divisor}\label{sec:sing-coh}
Consider a flat, projective morphism 
$\pi: \mc{X} \to \Delta$ of relative dimension $n$, 
smooth over $\Delta^*$ and the central fiber $\mc{X}_0$ is a reduced simple normal crossings divisor. 
Note that, the cycle class map $\mr{cl}_p$ from $A^p(\mc{X}_0)$ to 
 $\mr{Gr}^W_{2p} H^{2p}(\mc{X}_0,\mb{Q})$ factors through 
 \[H^{2p}_{\mr{Hdg}}(\mc{X}_0):= F^p \mr{Gr}^W_{2p} H^{2p}(\mc{X}_0,
 \mb{C}) \cap \mr{Gr}^W_{2p} H^{2p}(\mc{X}_0,\mb{Q}).\]
 This follows from comparing the Mayer-Vietoris sequences associated to the 
 operational Chow group (see Theorem \ref{th:Chow}) and that of the cohomology groups. 
 Denote by $H^{2p}_A(\mc{X}_0)$ the image of the cycle class map $\mr{cl}_p$.
We prove:

\begin{thm}\label{th01}
 Suppose that $\pi$ is a Mumford-Tate family. Let $p$ be a positive integer such that a general fiber of $\pi$ satisfies $\mr{HC}(p)$.
 If the intersection of any two irreducible component of $\mc{X}_0$ satisfies $\mr{HC}(p-1)$,
 then $\mc{X}_0$ satisfies $\mr{SHC}(p)$. 
\end{thm}

\begin{proof}
By Proposition \ref{prop04} we have a morphism $\mr{sp}_{_A}$ from $H^{2p}_A(\mc{X}_0)$ to $H^{2p}_A(\mc{X}_\infty)$ given by the composition:
\[\mr{sp}_{A}: H^{2p}_A(\mc{X}_0) \hookrightarrow H^{2p}_{\mr{Hdg}}(\mc{X}_0)
 \xrightarrow{\mr{sp}} H^{2p}_{\mr{Hdg}}(\mc{X}_\infty)^{\mr{inv}} \cong 
 H^{2p}_A(\mc{X}_\infty).\]
We claim that $\mr{sp}_{_A}$ is surjective. 
Recall from Definition \ref{defi:lim-alg}, $H^{2p}_A(\mc{X}_\infty)$ is generated as a $\mb{Q}$-vector space by classes $\gamma_{_{\mc{Z}}}$ where $\mc{Z} \subset \mc{X}_{\Delta^*}$ is a $\Delta^*$-flat closed 
subscheme of relative dimension 
$n-p$. 

Denote by $\ov{\mc{Z}}$ the closure of $\mc{Z}$ in $\mc{X}$ and $X_1,...,X_r$ be the irreducible components of $\mc{X}_0$.
By \cite[\S $6.1$]{fult}, the intersection product $\ov{\mc{Z}}.X_i$ of $\ov{\mc{Z}}$
with $X_i$ is of codimension $p$ in $X_i$.
Denote by $\gamma_i \in H^{2p}(X_i,\mb{Q})$ the cohomology class of 
the intersection product $\ov{\mc{Z}}.X_i$ for $1 \le i \le r$.
By the associativity of intersection product 
(see \cite[Proposition $8.1.1$ or Proposition $8.3$]{fult}), for any pair of integers $1 \le i <j \le r$, the image of 
$\gamma_i$ (resp. $\gamma_j$) under the restriction morphisms
from $H^{2p}(X_i, \mb{Q})$ (resp. $H^{2p}(X_j, \mb{Q})$)
to $H^{2p}(X_i \cap X_j, \mb{Q})$ coincides.
By the Mayer-Vietoris sequence, there exists an algebraic cohomology class 
$\gamma \in H^{2p}_A(X)$ such that the image of $\gamma$ under 
the restriction morphism from $H^{2p}_A(X)$ to $H^{2p}_A(X_i)$
is $\gamma_i$ for $1 \le i \le r$. In other words, the cohomology class
of $\ov{\mc{Z}}$ in $H^{2p}(\mc{X}, \mb{Q})$ (see \cite[\S B.$2.9$]{pet})
pulls back to $\gamma$ in $H^{2p}(\mc{X}_0, \mb{Q})$ and to 
the cohomology class $[\mc{Z} \cap \mc{X}_t] \in 
H^{2p}(\mc{X}_t, \mb{Q})$ over $\mc{X}_t$, for any $t \in \Delta^*$.
 Hence, $\mr{sp}_A(\gamma_{_\mc{Z}})=\gamma$. This proves our claim.

%

Denote by $X(2)$ the disjoint union of all the intersections of any two distinct irreducible components of $\mc{X}_0$.
By Proposition \ref{prop01}, the kernel of the specialization morphism
\[\mr{Gr}^W_{2p}H^{2p}(\mc{X}_0, \mb{Q}) = E_2^{0,2p} \xrightarrow{
 \mr{sp}}\, ^{\infty}E_2^{0,2p}=\mr{Gr}^W_{2p}H^{2p}(\mc{X}_\infty, \mb{Q}) \]
is isomorphic to the image of the Gysin morphism from 
$H^{2p-2}(X(2), \mb{Q})$ to $H^{2p}(X, \mb{Q})$ (as $X(2)$ is non-singular, 
$H^{2p-2}(X(2),\mb{Q})$ has a pure Hodge structure of weight $2p-2$).
By assumption, every irreducible component of 
$X(2)$ satisfies HC$(p-1)$.
Then, we get the following commutative diagram of exact 
sequences:
\[\begin{diagram}
   H^{2p-2}_A(X(2))&\rTo&H^{2p}_A(\mc{X}_0)&\rTo^{\mr{sp}_{_A}}&H^{2p}_A(\mc{X}_\infty)&\rTo&0\\
   \dTo^{\cong}&\circlearrowleft&\dInto&\circlearrowleft&
   \dTo^{\cong}\\
   H^{2p-2}_{\mr{Hdg}}(X(2))&\rTo&H^{2p}_{\mr{Hdg}}(\mc{X}_0)&\rTo^{\mr{sp}}&H^{2p}_{\mr{Hdg}}(\mc{X}_\infty)^{\mr{inv}}&
     \end{diagram}\]
By diagram chase (or using four lemma for the diagram of
exact sequences), we conclude that the middle 
morphism from $H^{2p}_A(\mc{X}_0)$ to $H^{2p}_{\mr{Hdg}}(\mc{X}_0)$ is surjective, hence an isomorphism.
In other words, $\mc{X}_0$ satisfies $\mr{SHC}(p)$. This proves the theorem.
\end{proof}

\subsection{Application to $A_n$ singularities}
Recall, an $A_n$ singularity of dimension $m$, which we denote by $A_n(m)$, is a hypersurface in $\mb{C}^{m+1}$
defined (analytic) locally by the equation  \cite{bur2}:
\[X_1^2+X_2^{n+1}+X_3^2+...+X_{m+1}^2=0,\, \, n \ge 1.\]
The blow-up of the singularity at the origin 
can be covered by affine schemes of the form 
\begin{equation}\label{eq04}
    \bigcup_{i \not= 2}^{m+1} \Spec\left( \frac{\mb{C}[X_i, \frac{X_1}{X_i},...,\frac{X_{m+1}}{X_i}]}{ \sum\limits_{j \not= 2}^{m+1} \left(\frac{X_j}{X_i}\right)^2+X_i^{n-1}\left(\frac{X_2}{X_i}\right)^{n+1}} \right) \cup 
   \Spec\left( \frac{\mb{C}[X_2, \frac{X_1}{X_2}, \frac{X_3}{X_2},..., \frac{X_{m+1}}{X_2}]}{ \sum\limits_{j \not= 2}^{m+1} \left(\frac{X_j}{X_2}\right)^2+X_2^{n-1} }\right)
\end{equation}
Note that, if $n \ge 3$ the first term defines a non-singular variety and the second term has $A_{n-2}(m)$-singularity. 
If $n \in \{1, 2\}$, then the blow-up is non-singular. The exceptional divisor in $\Spec(\mb{C}[X_i,\frac{X_1}{X_i},...,\frac{X_{m+1}}{X_i}])$ is defined by the intersection of $X_i=0$ and 
the zero locus of the quadratic polynomial 
\begin{eqnarray*}
     \sum_{j \not= 2} (X_j/X_i)^2,& \mbox{ if } n \ge 2,\\
     1+\sum_{j \not= 2} (X_j/X_i)^2,& \mbox{ if } n=1.
\end{eqnarray*}
This implies, if $n=1$, then the exceptional divisor is non-singular quadratic hypersurface in $\mb{P}^m$. If $n \ge 2$ and 
$i \not=2$, then the exceptional divisor is again non-singular, defined by a quadratic polynomial. If $i=2$ and $n \ge 2$, then 
the exceptional divisor is a cone over a non-singular quadric. To summarize, 
given a projective variety $X$ of dimension $m$ with at worst $A_n(m)$-singularities, there exists a sequence of blow-ups
\begin{equation}\label{eq:blow}
    \wt{X}:=X_N \xrightarrow{\phi_N} X_{N-1} \xrightarrow{\phi_{N-1}} ... \xrightarrow{\phi_2} X_1 \xrightarrow{\phi_1} X
\end{equation}
  where $\phi_i$ is the blow-up of $X_{i-1}$ at a singular point and $\wt{X}$ is non-singular. 
  Moreover, the exceptional divisor associated to the blow-up map $\phi_i$ is either a projective cone over a quadric hypersurface in $\mb{P}^{m-1}$
  or a  non-singular quadric hypersurface in $\mb{P}^{m}$.
  This implies that the exceptional divisor $E$ associated to the composed blow-up map from $X_N$ to $X$ is a disjoint union of 
  algebraic subsets of the form $E_1 \cup E_2 \cup ... \cup E_r$ where $E_i$ is a 
  blow-up of a projective cone over a quadric hypersurface in $\mb{P}^{m-1}$ or a non-singular quadric in $\mb{P}^{m}$, $E_i \cap E_j=\emptyset$ for $|i-j|>1$ and 
  $E_i \cap E_{i+1}$ is a non-singular quadric hypersurface in $\mb{P}^{m-1}$.

\begin{lem}\label{lem:van-qua}
Let $m>1$ and $Q$ be a projective cone over a non-singular 
quadric hypersurface in $\mb{P}^{m}$. Denote by $\wt{Q}$ the blow-up of $Q$ at the singular point.
Then,  the cycle class 
map from $A^j(\wt{Q})$ to $H^{2j}_{\mr{Hdg}}(\wt{Q})$ is an isomorphism for all $j>0$. 
Furthermore, if $m$ is even, then $A^j(Q) = \mb{Q}\mb{L}^j=H^{2j}_{\mr{Hdg}}(Q)$ for all $j>0$, where $\mb{L}$ is a hyperplane section on $Q$. 
\end{lem}

\begin{proof}
    Denote by $U$ the non-singular locus of $Q$ i.e., $U=Q\backslash \{x\}$, where $x \in Q$ is the vertex of the cone $Q$.
    Then, $U$ is an $\mb{A}^1$-bundle over a non-singular quadric in $\mb{P}^{m}$, say $Q^c$. This means $Q^c$ is a deformation retract of $U$.
    Hence,     their cohomology groups coincide.
    Therefore, 
    \[A_j(Q) \cong A_j(U) \cong A_{j-1}(Q^c)\, \, \mbox{ and } H_{2j}^{\mr{BM}}(Q) \cong H_{2j}^{\mr{BM}}(U) \cong H_{2j-2}^{\mr{BM}}(Q^c)\mbox{ for all } j>0\]
   Since $Q^c$ is a non-singular quadric, this means that the cycle class map from $A_j(Q) \cong A_{j-1}(Q^c)$ to $H_{2j}(Q,\mb{Q}) \cong H_{2j-2}(Q^c,\mb{Q})$
   is an isomorphism. 
    Denote by $E$ the exceptional divisor of the blow-up $\wt{Q}$ of $Q$. Note that, $E$ is isomorphic to $Q^c$. 
    We then have the following diagram of exact sequences:
    \begin{equation}\label{eq:diag7}
        \begin{diagram}
A_j(E) &\rTo^{\phi_1}& A_j(\wt{Q})& \rTo& A_j(Q)& \rTo& 0\\
\dTo^{\mr{cl}}_{\cong}&\circlearrowleft&\dTo^{\mr{cl}}&\circlearrowleft&\dTo^{\mr{cl}}_{\cong}\\
H_{2j}(E, \mb{Q})&\rTo^{\phi_2}&H_{2j}(\wt{Q},\mb{Q})&\rTo&H_{2j}(Q,\mb{Q})&\rTo&0
    \end{diagram}
    \end{equation}
    where the surjectivity on the right for the bottom row follows from the fact that 
    the odd cohomology of non-singular quadrics is zero (i.e., $H_{2j-1}(E,\mb{Q})=0$),  
    the first vertical isomorphism follow from the fact that $E$ is a non-singular quadric and the last vertical isomorphism was 
    observed above. By diagram chase, we conclude that the cycle class map from $A_j(\wt{Q})$ to $H_{2j}(\wt{Q},\mb{Q})$ is an isomorphism.
    Since cap product commutes with the cycle class map, this implies that $A^{m-j}(\wt{Q})$ to $H^{2m-2j}(\wt{Q},\mb{Q})$ is an isomorphism.
    This proves the first part of the lemma.

    Now suppose that $m$ is even. Then, $Q^c$ is an odd dimensional quadric. Hence, $A_j(Q^c) \cong \mb{Q}$ for all $j \ge 0$, generated by 
    the $(m-1-j)$-th power of the hyperplane section. 
    This implies that the morphism $\phi_1$ in \eqref{eq:diag7} is injective and $A_j(Q) \cong \mb{Q}$. 
    Using the exact sequence in the top row of \eqref{eq:diag7} we conclude that $A_j(\wt{Q}) \cong \mb{Q}^{\oplus 2}$
    for all $j> 0$. Since $\wt{Q}$ is non-singular, this implies $A^j(\wt{Q}) \cong \mb{Q}^{\oplus 2}$ for $0<j<m$. 
    Using Theorem \ref{th:Chow}, we have the exact sequence:
    \begin{equation}\label{eq:quad-op}
        0 \to A^j(Q) \to A^j(\wt{Q}) \to A^j(E) \to 0
    \end{equation}
    where the surjectivity on the right follows from $A^j(E) \cong \mb{Q}$ generated by the $j$-th power of a hyperplane section. 
    This implies $A^j(Q) = \mb{Q}\mb{L}^j$, where $\mb{L}$ is a hyperplane section on $Q$. Replacing the operational Chow group with the cohomology 
    group, we similarly observe that $H^{2j}_{\mr{Hdg}}(Q)=\mb{Q}\mb{L}^j$.     This proves the lemma.
\end{proof}

\begin{prop}\label{prop:ind}
    Let $X_i$ be as in \eqref{eq:blow}. Suppose that $m$ is odd and the exceptional divisor $Q_i$ associated to the blow-up $\phi_i$
    is a projective cone over a non-singular quadric in $\mb{P}^{m-1}$. If $X_{i-1}$ satisfies $\mr{SHC}(p)$, then so does $X_i$.
\end{prop}

\begin{proof}
   Consider the blow-up $X_i$ of $X_{i-1}$. Denote by $\phi_Q:Q_i \hookrightarrow X_i$ the natural inclusion. 
   Take any $\alpha \in H^{2p}_{\mr{Hdg}}(X_i)$. We need to find an element $\alpha_{\mr{alg}} \in A^p(X_i)$ such that 
   $\mr{cl}(\alpha_{\mr{alg}})=\alpha$.
   Since $m-1$ is even, Lemma \ref{lem:van-qua} implies that 
   the cycle class map from $A^p(Q_i)$ to $H^{2p}_{\mr{Hdg}}(Q_i)$ is an isomorphism.
   Hence, there exists $\beta \in A^p(Q_i)$ such that $\phi_Q^*(\alpha)=\mr{cl}(\beta)$. 
   Moreover, $A^p(Q_i) \cong \mb{Q}$ implies that the restriction morphism $\phi_Q^*$ from $A^p(X_i)$ to $A^p(Q_i)$ is surjective.
   As a result, there exists $\alpha_1 \in A^p(X_i)$ such that $\phi_Q^*(\alpha_1)=\beta$.
   If $\mr{cl}(\alpha_1)=\alpha$, then take $\alpha_{\mr{alg}}:=\alpha_1$ and we are done. 
   If $\mr{cl}(\alpha_1) \not= \alpha$, then $\phi_Q^*(\alpha-\mr{cl}(\alpha_1))=0$. 
   Consider the exact sequence, 
   \begin{equation}\label{eq:quad-coh}
       0 \to H^{2p}_{\mr{Hdg}}(X_{i-1}) \xrightarrow{\phi_i^*} H^{2p}_{\mr{Hdg}}(X_i) \xrightarrow{\phi_Q^*} H^{2p}_{\mr{Hdg}}(Q_i)
   \end{equation}
   Since $X_{i-1}$ satisfies $\mr{SHC}(p)$, there exists $\alpha_2 \in A^p(X_{i-1})$ such that $\phi_i^*(\mr{cl}(\alpha_2))=\alpha-\mr{cl}(\alpha_1)$.
   Using the functoriality of the cycle class map,  we conclude that $\alpha=\mr{cl}(\phi_i^*\alpha_2)+\mr{cl}(\alpha_1)$. 
   In other words, for $\alpha_{\mr{alg}}:= \phi_i^*\alpha_2+\alpha_1$, we have $\mr{cl}(\alpha_{\mr{alg}})=\alpha$. 
   This proves the proposition.   
\end{proof}

  \begin{thm}\label{thm:final}
Let $Y$ be a smooth, projective variety of dimension $2m$ satisfying the Hodge conjecture, 
$X \subset Y$ very ample divisor of $Y$ with at worst $A_n(2m-1)$-singularities. Then, 
\begin{enumerate}
    \item $X$ satisfies $\mr{SHC}(p)$ for all $p>0$ i.e., $X$ satisfies the singular Hodge conjecture,
    \item there exists a resolution of singularities $\wt{X}$ of $X$ (by recursively blowing-up the singular locus)
        such that $\wt{X}$ satisfies $\mr{HC}(p)$ for all $p \ge m$.
        \item Moreover, if $n$ is even, then $\wt{X}$ satisfies $\mr{HC}(p)$ for all $p>0$ i.e., $\wt{X}$ satisfies the Hodge conjecture.        
\end{enumerate}
\end{thm}

  \begin{proof}
       Let $X \subset Y$ be as above. Denote by $\wt{X}$ the resolution of $X$ obtained by 
       recursively blowing up the singular locus of $X$. 
       By the Lefschetz hyperplane theorem, we have for $p<m$:
       \[H^{2p}_{\mr{Hdg}}(X) \cong H^{2p}_{\mr{Hdg}}(Y) \cong H^{2p}_A(Y) \to H^{2p}_A(X) \hookrightarrow H^{2p}_{\mr{Hdg}}(X)\]
       is the identity morphism, where the second isomorphism follows from the hypothesis on $Y$. 
       This implies $H^{2p}_A(X) \cong H^{2p}_{\mr{Hdg}}(X)$. Hence, $X$ satisfies $\mr{SHC}(p)$ for $p<m$. 
      
        We now use Theorem \ref{th01} above to prove that $X$ satisfies $\mr{SHC}(p)$ for all $p \ge m$. 
     Embed $X$ in a family of odd dimensional hypersurfaces of $Y$
     \[\pi:\mc{X} \to \Delta\]
     smooth over $\Delta^*$ and central fiber $X$. Then, after base change of the family by the morphism $t \mapsto t^2$, 
     the resulting family has $A_n(2m)$-singularities. As discussed above, we can resolve the singularities by 
     successive blow-ups along the singular locus 
     of the central fiber. In particular, we get a semi-stable family:
     $\pi^{\mr{ss}}: \mc{X}^{\mr{ss}} \to \Delta$
     such that the central fiber $\mc{X}^{\mr{ss}}_0$ is a reduced simple normal crossings divisor of the form 
     $\mc{X}_0^{\mr{ss}}=\wt{X} \cup E$,  where each irreducible component of $E$ is either a non-singular quadric in 
     $\mb{P}^{2m}$ or the blow-up at the vertex of a projective cone over a non-singular quadric in $\mb{P}^{2m-1}$ (the projective cone is of 
     dimension $2m-1$). 
     By Theorem \ref{thm:mt-exa}, $\pi^{\mr{ss}}$ is a Mumford-Tate family. By Lefschetz hyperplane theorem, $\mc{X}_t$ satisfies $\mr{HC}(p)$
     for all $p$ and $t \in \Delta^*$.
     Then, Theorem \ref{th01} implies that $\mc{X}_0^{\mr{ss}}$ satisfies $\mr{SHC}(p)$ for all $p >0$.
     Denote by $E_1,...,E_r$ the irreducible components of $E$. 
     Using Theorem \ref{th:Chow} we have an exact sequence of the form:
     \begin{equation}\label{eq:chow-1}
         0 \to A^p(X) \xrightarrow{\phi^*} A^p(\mc{X}_0^{\mr{ss}}) \xrightarrow{\phi_E^*} \bigoplus_i A^p(E_i)
     \end{equation}
     where $\phi$ (resp. $\phi_E$) is the natural morphism from $\mc{X}_0^{\mr{ss}}$ (resp. $\coprod E_i$) to $X$ (resp. $\mc{X}_0^{\mr{ss}}$). 
     We now show that for any $\alpha \in H^{2p}_{\mr{Hdg}}(X)$, there exists $\alpha_{\mr{alg}} \in A^p(X)$ such that 
     $\mr{cl}(\alpha_{\mr{alg}})=\alpha$. 
     Indeed, since $\mc{X}_0^{\mr{ss}}$ satisfies $\mr{SHC}(p)$, there exists $\beta \in A^p(\mc{X}_0^{\mr{ss}})$ such that 
     $\mr{cl}(\beta)=\phi^*(\alpha)$. Consider the exact sequence (use \cite[Corollary-Definition $5.37$]{pet} along with the Mayer-Vietoris long exact sequence):
     \begin{equation}\label{eq:ex-1}
         0 \to H^{2p}_{\mr{Hdg}}(X) \xrightarrow{\phi^*} H^{2p}_{\mr{Hdg}}(\mc{X}_0^{\mr{ss}}) \xrightarrow{\phi_E^*} \bigoplus_i H^{2p}_{\mr{Hdg}}(E_i)
     \end{equation}
     This implies $\phi_E^*(\phi^*\alpha)=0$. Since the cycle class map from $A^p(E_i)$ to $H^{2p}_{\mr{Hdg}}(E_i)$ is injective (Lemma \ref{lem:van-qua}),
     this implies $\phi_E^*(\beta)=0$. Hence, there exists $\alpha_{\mr{alg}} \in A^p(X)$ such that $\phi^*(\alpha_{\mr{alg}})=\beta$.
     By the injectivity of $\phi^*$ in \eqref{eq:ex-1} along with the functoriality of the cycle class map, 
     we conclude that $\mr{cl}(\alpha_{\mr{alg}})=\alpha$. This proves our claim. Hence, 
     $X$ satisfies $\mr{SHC}(p)$ for all $p \ge 0$, proving $(1)$.

     We now show that $\wt{X}$ satisfies $\mr{HC}(p)$. For simplicity, we assume that 
     $X$ has exactly one $A_n(2m-1)$ singularity. Then, there is a sequence of blow-ups as in \eqref{eq:blow}
     such that the exceptional divisor associated to $\phi_i$ for $i<N$ is a projective cone over a non-singular quadric in $\mb{P}^{2m-2}$
     and the exceptional divisor $Q_N$ associated to $\phi_N$ is either a non-singular quadric in $\mb{P}^{2m-1}$ or 
     a projective cone over a non-singular quadric in $\mb{P}^{2m-2}$. 
     Since $X$ satisfies $\mr{SHC}(p)$, Proposition \ref{prop:ind} implies that $X_i$ satisfies $\mr{SHC}(p)$ for all $p>0$ and $i<N$.
     If $Q_N$ is a projective cone over a quadric, then Proposition \ref{prop:ind} again implies that $X_N=\wt{X}$ satisfies 
     $\mr{HC}(p)$ for all $p>0$. This is the case if $n$ is even (blow-up of an $A_n$ singularity gives an $A_{n-2}$ singularity). This proves $(3)$. 

   To prove $(2)$, we need to show that if $n$ is odd, then $\wt{X}$ satisfies $\mr{HC}(p)$ for all $p \ge m$. Once again for simplicity, 
   we assume that $X$ has only one singular point. As observed above, $X_{N-1}$ satisfies $\mr{SHC}(p)$ for all $p>0$ and the exceptional divisor 
   $Q_N$ associated to the final blow-up $\phi_N$ is a non-singular quadric in $\mb{P}^{2m-1}$. Since $\dim(Q_N)=2m-2$, for $p \ge m$, 
  $A^p(Q_N)=\mb{Q}\mb{L}^p$, where $\mb{L}$ is a hyperplane section in $Q_N$. This implies that the restriction morphism $\phi_Q^*$ from $A^p(\wt{X})$ to 
  $A^p(Q_N)$ is surjective i.e., we have the short exact sequence:
  \[0 \to A^p(X_{N-1}) \xrightarrow{\phi_N^*} A^p(\wt{X}) \xrightarrow{\phi_Q^*} A^p(Q_N) \to 0\, \mbox{ for all } p \ge m.\]
  Then, arguing exactly as in Proposition \ref{prop:ind}, we observe that for any $\alpha \in H^{2p}_{\mr{Hdg}}(\wt{X})$ there exists 
  $\alpha_{\mr{alg}} \in A^p(\wt{X})$ such that $\mr{cl}(\alpha_{\mr{alg}})=\alpha$ i.e., $\wt{X}$ satisfies $\mr{HC}(p)$ for all $p \ge m$.
  This proves $(2)$ and hence the theorem.
      \end{proof}

We denote by $F^{-p}W_{-2p}H^{\mr{BM}}_{2p}(X,\mb{Q}):=F^{-p}H^{\mr{BM}}_{2p}(X,\mb{C}) \cap W_{-2p}H_{2p}^{\mr{BM}}(X,\mb{Q})$. We prove:

 \begin{cor}\label{cor:jann}
     Setup as in Theorem \ref{thm:final}. If $n$ is even, then $X$ satisfies Jannsen's homological Hodge conjecture i.e., the cycle class map from $A_p(X)$ to $F^{-p}W_{-2p}H_{2p}^{\mr{BM}}(X,\mb{Q})$
     is surjective for all $p>0$.
 \end{cor}

\begin{proof}
By Theorem \ref{thm:final}, $\wt{X}$ satisfies the Hodge conjecture. Since the cycle class commutes with cap product, this implies 
\[\mr{cl}:A_p(\wt{X}) \to F^{-p}W_{-2p} H_{2p}^{\mr{BM}}(\wt{X})\, \mbox{ is surjective, for all } p > 0.\]
Denote by $E$ the exceptional divisor for the blow-up map from $\wt{X}$ to $X$.
 Since $X$ has only isolated singularities, $H_{2p}^{\mr{BM}}(X) \cong H_{2p}^{\mr{BM}}(X_{\mr{sm}})$, where $X_{\mr{sm}}$ is the smooth locus in $X$
 and we have a homology long 
 exact sequence:
 \[ ... \to H_{2p}^{\mr{BM}}(\wt{X},\mb{Q}) \xrightarrow{\phi_*} H_{2p}^{\mr{BM}}(X,\mb{Q}) \to  H_{2p-1}^{\mr{BM}}(E, \mb{Q}) \to ...\]
 Since $H_{2p-1}^{\mr{BM}}(E,\mb{Q})$ has weights between $-(2p-1)$ and $0$, the natural map from 
 $F^{-p}W_{-2p} H_{2p}^{\mr{BM}}(\wt{X}, \mb{Q})$ to $F^{-p}W_{-2p} H_{2p}^{\mr{BM}}(X, \mb{Q})$ is surjective. Then, we have the following commutative 
 diagram:
 \[\begin{diagram}
 A_p(\wt{X})&\rTo^{\phi_*}&A_p(X)\\
 \dOnto^{\mr{cl}}&\circlearrowleft&\dTo^{\mr{cl}}\\
 F^{-p}W_{-2p} H_{2p}^{\mr{BM}}(\wt{X}, \mb{Q})&\rOnto^{\phi_*}&F^{-p}W_{-2p} H_{2p}^{\mr{BM}}(X, \mb{Q})     
 \end{diagram}.\]
 By the commutativity of the diagram, this implies the right hand side vertical arrow is surjective for all $p>0$. In other words, $X$
 satisfies Jannsen's homological Hodge conjecture. This proves the corollary.
 \end{proof}

As before, given a smooth, projective variety $W$, we denote by $H^{2i}_A(W)$ the image of the cycle class map.
Recall, the Lefschetz Standard conjecture A:

\begin{lsa}
    Given a smooth, projective variety $W$ of dimension $2m-1$, the natural morphism from 
$H^{2m-2-2i}_A(W,\mb{Q})$ to $H^{2m+2i}_A(W,\mb{Q})$ induced by cup-product 
with a hyperplane section is an isomorphism for all $i \ge 0$.
\end{lsa}
 We prove:

\begin{cor}\label{cor:n-odd}
   Setup as in Theorem \ref{thm:final}. If $n$ is even and the Lefschetz standard conjecture A holds true, then 
   $\wt{X}$ satisfies the Hodge conjecture.
\end{cor}

\begin{proof}
   By the Hard Lefschetz theorem, the 
    cup-product by a hyperplane section $\mb{L}$ induces an isomorphism \[\cup \mb{L}: H^{2m-2-2i}_{\mr{Hdg}}(\wt{X},\mb{Q}) \xrightarrow{\sim} H^{2m+2i}_{\mr{Hdg}}(\wt{X},\mb{Q})\, \mbox{ for all }\, i \ge 0.\]
    By Theorem \ref{thm:final} we know that $\wt{X}$ satisfies $\mr{HC}(p)$ for all $p \ge m$ i.e., $H^{2m+2i}_A(\wt{X}) \cong H^{2m+2i}_{\mr{Hdg}}(\wt{X})$
    for all $i \ge 0$. Then, Lefschetz standard conjecture A implies that $H^{2m-2-2i}_A(\wt{X}) \cong H^{2m-2-2i}_{\mr{Hdg}}(\wt{X})$
    for all $i \ge 0$. Hence, $\wt{X}$ satisfies the Hodge conjecture. This proves the corollary.
\end{proof}

\appendix

\section{Commutativity properties of monodromy}\label{app:mon}
We now study some commutativity properties of the monodromy operators and their logarithms.
Setup as in \S \ref{se01}. Denote by $T$ the monodromy operator and $N$ the logarithm of the unipotent part of $T$ as defined in \S \ref{sec:lhf}.

\begin{lem}\label{lem:mon}
We have the following:
\begin{enumerate}
    \item cup-product commutes with monodromy i.e., 
    \[T(\alpha \cup \beta)=T(\alpha) \cup T(\beta)\, \mbox{ for any } \alpha \in H^i(\mc{X}_\infty, \mb{Q}) \mbox{ and } \beta \in H^j(\mc{X}_\infty, \mb{Q}). \]
    \item pull-back commutes with monodromy i.e., given two smooth, projective families:
    \[ \pi_1: \mc{X} \to \Delta^*, \, \pi_2: \mc{Y} \to \Delta^*\, \mbox{ and } f: \mc{X} \to \mc{Y}\, \mbox{ satisfying } \pi_1=\pi_2 \circ f,\]
    we have $T^{(\pi_1)} \circ f^*(\alpha)=f^* \circ T^{(\pi_2)}(\alpha)$ for any $\alpha \in H^*(\mc{Y}_\infty, \mb{Q})$,
    \item K\"{u}nneth decomposition commutes with monodromy i.e., given two smooth, projective families $\pi_1, \pi_2$ as in $(2)$ we have 
    \[T^{(\pi_1 \times \pi_2)} \circ \mr{KD}_i (\alpha \otimes \beta)=\mr{KD}_i \circ \left(T^{(\pi_1)} \times T^{(\pi_2)}\right) (\alpha \otimes \beta)\]
    for any $\alpha \in H^i(\mc{X}_\infty, \mb{Q}),\, \beta \in H^{m-i}(\mc{Y}_\infty, \mb{Q})$ and 
    $\mr{KD}_i$ is the composed (K\"{u}nneth decomposition) morphism  {\small 
    \[ \mr{KD}_i: H^i(\mc{X}_\infty, \mb{Q}) \otimes H^{m-i}(\mc{Y}_{\infty}, \mb{Q}) \xrightarrow{\pr_1^* \otimes \pr_2^*} 
        H^i(\mc{X}_\infty \times_{\mf{h}} \mc{Y}_\infty, \mb{Q}) \otimes H^{m-i}(\mc{X}_\infty \times_{\mf{h}} \mc{Y}_{\infty}, \mb{Q}) \xrightarrow{\cup} 
        H^{m}(\mc{X}_\infty \times_{\mf{h}} \mc{Y}_\infty, \mb{Q}),\]}
    $\pr_1, \pr_2$ are natural projection maps from $\mc{X}_\infty \times_{\mf{h}} \mc{Y}_\infty$ to $\mc{X}_\infty, \mc{Y}_\infty$, respectively.
    \end{enumerate}
\end{lem}

\begin{proof}
    Since pull-back morphism commutes with cup-product, \eqref{eq:mon} implies that the monodromy operator commutes with cup-product i.e., 
    we have $(1)$.

    We now prove $(2)$. Denote by $f_\infty: \mc{X}_\infty \to \mc{Y}_\infty$ be the map induced by $f$ and 
    \[h_1: \mc{X}_\infty \to \mc{X}_\infty,\, \, \, h_2:\mc{Y}_\infty \to \mc{Y}_\infty\]
    defined by 
     $(z,u)$ maps to $(z,u+1)$. Recall from \eqref{eq:mon} that 
    $T^{(\pi_1)} = h_1^*$ and $T^{(\pi_2)}= h_2^*$. Clearly, $f_\infty \circ h_1 = h_2 \circ f_\infty$. Therefore, for any 
  $\alpha \in H^*(\mc{Y}_\infty, \mb{Q})$ we have 
 \[T^{(\pi_1)} \circ f_\infty^*(\alpha)=h_1^* \circ f_\infty^*(\alpha)=f_\infty^* \circ h_2^*(\alpha)=f_\infty^* \circ T^{(\pi_2)}(\alpha).\] This proves $(2)$.

 Since the morphism $\mr{KD}_i$ is a composition of a pull-back morphism with the cup-product map, $(3)$ follows from $(1)$ and $(2)$. This proves the lemma.
\end{proof}

\begin{cor}\label{cor:mon}
    The following holds true:
    \begin{enumerate}
        \item Let $\pi$ be a smooth, projective family as in Lemma \ref{lem:mon}(1). Suppose that $T$ acts trivially on $H^m(\mc{X}_\infty, \mb{Q})$.
        Then, for any $\alpha \in H^m(\mc{X}_\infty, \mb{Q})$ and $\beta \in H^i(\mc{X}_\infty, \mb{Q})$, we have 
        \[N(\alpha \cup \beta)=\alpha \cup N(\beta),\]
        for any $i$,
        \item Let $\pi_1, \pi_2$ and $f$ be as in Lemma \ref{lem:mon}(2). Then, for any $\alpha \in H^i(\mc{X}_\infty, \mb{Q})$, we have 
        $N^{(\pi_1)} \circ f^*(\alpha)=f^* \circ N^{(\pi_2)}(\alpha)$,
        \item Let $\pi_1$, $\pi_2$ be as in Lemma \ref{lem:mon}(3). 
        Suppose that $T^{(\pi_1)}$ acts trivially on $H^i(\mc{X}_\infty, \mb{Q})$. Then, 
        for any $\alpha \in H^i(\mc{X}_\infty, \mb{Q})$ and $\beta \in H^{m-i}(\mc{Y}_\infty, \mb{Q})$, we have 
        \[N^{(\pi_1 \times \pi_2)}(\mr{KD}_i(\alpha \otimes \beta))=\mr{KD}_i(\alpha \otimes N^{(\pi_2)}(\beta))\]
        where $\mr{KD}_i$ is as in Lemma \ref{lem:mon}, for any $i$.
    \end{enumerate}
\end{cor}

\begin{proof}
(1): We first claim that it is sufficient to prove that for $\alpha \in H^m(\mc{X}_\infty, \mb{Q})$ and $\beta \in H^i(\mc{X}_\infty, \mb{Q})$, we have 
\begin{equation}\label{eq:cor-mon1}
    (T-\mr{id})(\alpha \cup \beta)=\alpha \cup (T-\mr{id})(\beta)
\end{equation}
    Indeed, if this is true then for any $n \ge 1$
    \[(T-\mr{id})^n(\alpha \cup \beta)=(T-\mr{id})^{n-1}((T-\mr{id})(\alpha \cup \beta))=(T-\mr{id})^{n-1}(\alpha \cup (T-\mr{id})(\beta)).\]
    Applying recursion on $n \ge 1$, we get that $(T-\mr{id})^n(\alpha \cup \beta)=\alpha \cup (T-\mr{id})^n(\beta)$. Therefore,
    
    \begin{align*}
        N(\alpha \cup \beta)&=\sum\limits_{k \ge 1} (-1)^{k+1} \frac{(T-\mr{id})^k(\alpha \cup \beta)}{k}=\sum\limits_{k \ge 1} (-1)^{k+1} \frac{\alpha \cup (T-\mr{id})^k(\beta)}{k}=\\ 
        &=\alpha \cup \sum\limits_{k \ge 1} (-1)^{k+1} \frac{\alpha \cup (T-\mr{id})^k(\beta)}{k}=\alpha \cup N(\beta)
    \end{align*}
    This proves our claim. Now, 
    \[ (T-\mr{id})(\alpha \cup \beta)=T(\alpha \cup \beta)-\alpha \cup \beta = T(\alpha) \cup T(\beta)- \alpha \cup \beta=\alpha \cup T(\beta)-\alpha \cup \beta
    =\alpha \cup (T-\mr{id})\beta,\]
    where the second equality follows from Lemma \ref{lem:mon} and we have the third equality since $T$ acts trivially on $H^m(\mc{X}_\infty)$. This proves 
    the equality \eqref{eq:cor-mon1}, thereby proving (1).

    (2): Arguing similarly as in the first part of the proof of (1), it suffices to show that 
    \[(T^{(\pi_1)}-\mr{id}) \circ f^*(\alpha)=f^* \circ (T^{(\pi_2)}-\mr{id})(\alpha).\]
    But this is easy to show:
    \[ (T^{(\pi_1)}-\mr{id}) \circ f^*(\alpha)= T^{(\pi_1)}\circ f^*(\alpha)-f^*(\alpha)=f^* \circ T^{(\pi_2)}(\alpha)-f^*(\alpha) = f^* \circ (T^{(\pi_2)}-\mr{id})\alpha\]
    where the second equality follows from Lemma \ref{lem:mon}. This proves (2).

    (3): Arguing similarly as in the first part of the proof of (1), it suffices to show that 
    \[(T^{(\pi_1 \times \pi_2)}-\mr{id}) (\mr{KD}_i(\alpha \otimes \beta))=\mr{KD}_i(\alpha \otimes (T^{(\pi_2)}-\mr{id})\beta).\]
    But, this equality follows from Lemma \ref{lem:mon}. Indeed, 
    \begin{align*}
        & (T^{(\pi_1 \times \pi_2)}-\mr{id}) (\mr{KD}_i(\alpha \otimes \beta)) = T^{(\pi_1 \times \pi_2)}(\mr{KD}_i(\alpha \otimes \beta))-\mr{KD}_i(\alpha \otimes \beta)=\\&=\mr{KD}_i(T^{(\pi_1)}(\alpha) \otimes T^{(\pi_2)}(\beta))-\mr{KD}_i(\alpha \otimes \beta)
         = \mr{KD}_i(\alpha \otimes T^{(\pi_2)}(\beta)) - \mr{KD}_i(\alpha \otimes \beta)=\\
        &=\mr{KD}_i(\alpha \otimes (T^{(\pi_2)}-\mr{id})\beta).
    \end{align*}
 where the first equality of the second line follows from Lemma \ref{lem:mon} and the last equality follows from linearity of $\mr{KD}_i$. This proves 
 (3) and hence the corollary.
\end{proof}

\section{Operational Chow group}\label{app:op}
Operational Chow groups (see \cite[\S $17$]{fult}) are generalizations of the classical Chow group.
The advantage of the operational Chow group is that it commutes with arbitrary pull-back and 
satisfies the usual projection formula.
Both notions of Chow groups coincide over non-singular projective varieties.
Here we list some properties. 
Let $X$ be a reduced projective scheme (possibly singular/ reducible). Denote by 
 $A^p(X)$ the $p$-th operational Chow group and by $A_q(X)$ the usual Chow group (free abelian group 
 generated by rational equivalence classes of dimension $q$ irreducible subvarieties of $X$).
 
\begin{thm}\label{th:Chow}
 Let $X$ be a reduced projective scheme. Then, the following are true:
 \begin{enumerate}
  \item if $X$ is irreducible and $\phi: \wt{X} \to X$ is a resolution of singularities with exceptional 
  divisor, then we have the following exact sequences for every $p$:
  \begin{align}
      & 0 \to A^p(X) \to A^p(\wt{X}) \oplus A^p(X_{\mr{sing}}) \to A^p(E)\, \mbox{ and } \label{eq:op2}\\
      & A_p(E) \to A_p(X_{\mr{sing}}) \oplus A_p(\wt{X}) \to A_p(X) \to 0. \label{eq:op3}
  \end{align}
  \item if $X=X_1 \cup X_2$ is the union of two distinct reduced projective without common component,
  then the following sequences are exact for all $p$:
  \begin{align}
      & 0 \to A^p(X) \to A^p(X_1) \oplus A^p(X_2) \to A^p(X_1 \cap X_2)\ \, \mbox{ and }\label{eq:op4}\\
      & A_p(X_1 \cap X_2) \to A_p(X_1) \oplus A_p(X_2) \to A_p(X_1 \cup X_2) \to 0 \label{eq:op5}
  \end{align}
  \item {Projection formula:} Given any morphism $f:X \to Y$, the cap-product map:
  \[A^p(X) \otimes A_q(X) \xrightarrow{\cap} A_{q-p}(X)\, \mbox{ (same for Y) satisfies }\, f_*(f^*\beta \cap \alpha)=\beta \cap f_*(\alpha)\]
  for any $\beta \in A_*(Y)$ and $\alpha \in A^*(X)$
  \item {Duality:} if $X$ is smooth, projective of dimension $n$, then the cap product map:
  \[ \cap [X]: A^p(X) \to A_{n-p}(X)\, \mbox{ is an isomorphism}.\]
 \end{enumerate}
\end{thm}

\begin{proof}
 \begin{enumerate}
  \item The exactness of \eqref{eq:op2} follows from \cite[Theorem $2(iii)$ and \S $3.1.1$]{soug}.
  The exactness of \eqref{eq:op3} is proved in \cite[Example $1.8.1$]{fult}.
  \item The exactness of \eqref{eq:op4} follows from \cite[Theorem $2(iv)$ and \S $3.1.1$]{soug}.
  The exactness of \eqref{eq:op5} follows from \cite[Example $1.8.1$]{fult}.  
  \item See \cite[\S $17.3$]{fult}.
  \item See \cite[Corollary $17.4$]{fult}.
 \end{enumerate}
\end{proof}

\end{document}